\documentclass[12pt,twoside,reqno]{amsart}
\usepackage{amsfonts}
\usepackage{amssymb}
\usepackage[notcite,notref]{%showkeys
}

\numberwithin{equation}{section}

\newtheorem{teo}{Theorem}[section]
\newtheorem{prop}[teo]{Proposition}
\newtheorem{lem}[teo]{Lemma}
\newtheorem{cor}[teo]{Corollary}

\theoremstyle{definition}

\newtheorem{rem}[teo]{Remark}

\numberwithin{equation}{section}

\def\a{\alpha}
\def\b{\beta}
\def\l{\lambda}
\def\g{\gamma }
\def\o{\omega}
\def\R{\mathbb{R}}
\def\N{\mathbb{N}}
\def\Z{\mathbb{Z}}
\def\d{\delta}

\def\t{\theta}
\def\f{\varphi}
\def\s{\sigma}

\subjclass[2010]{Primary 46E35, 42B05,  42B10; \newline Secondary 46E30 }

\begin{document}

%%%%%%%%%%%%%%%%%%%%%%%%%%%%%%%%%%%%%%%%%%%%%%%%%%%%%%%%%%%%%%%%%%%%%%

\title[FRACTIONAL INTEGRALS]{FRACTIONAL INTEGRALS  AND FOURIER TRANSFORMS}

\author[V.I. Kolyada]{V.I. Kolyada}
\address{Department of Mathematics\\
Karlstad University\\
Universitetsgatan 1 \\
651 88 Karlstad\\
SWEDEN} \email{viktor.kolyada@gmail.com}

%%%%%%%%%%%%%%%%%%%%%%%%%%%%%%%%%%%%%%%%%%%%%%%%%%%%%%%%%%%%%%%%%%%%%%

\maketitle

\date{}

\maketitle

This  paper gives a short survey of some basic results related to estimates of fractional integrals and Fourier transforms.
It is closely adjoint to our previous survey papers \cite{K1998} and \cite{K2007}.
The main methods used in the paper are based on nonincreasing rearrangements. We give alternative proofs of some results.

We observe also that the paper
represents the mini-course given by the author at Barcelona University in October, 2014.

\vskip 8pt

\section{Nonincreasing rearrangements}
\vskip 4pt

Denote by $S_0(\mathbb{R}^n)$ the class of all measurable and
almost everywhere finite functions $f$ on $\mathbb{R}^n$ such that
for each $y>0$
\begin{equation*}
\lambda_f (y) \equiv | \{x \in \mathbb{R}^n : |f(x)|>y \}| <
\infty.
\end{equation*}

A non-increasing rearrangement of a function $f \in
S_0(\mathbb{R}^n)$ is a non-increasing function $f^*$ on
$\mathbb{R}_+ \equiv (0, + \infty)$ such that for any $y>0$
\begin{equation*}
|\{t\in \mathbb{R}_+ : f^*(t)>y\}|= \lambda_f (y).
\end{equation*}
We shall assume in addition that the rearrangement $f^*$ is left
continuous on $(0,\infty).$ Under this condition it is defined
uniquely by
$$
f^*(t)=\inf\{y>0: \lambda_f (y)<t\}, \quad 0<t<\infty.
$$
Besides, we have the equality
\begin{equation*}
f^*(t) = \sup_{|E|=t} \inf_{x \in E} |f(x)|. \,\quad
\end{equation*}
The following relation holds
\begin{equation} \label{sup}
\sup_{|E|=t} \int_E |f(x)| dx = \int_0^t f^*(u) du \,.
\end{equation}
In what follows we denote
\begin{equation*}
f^{**}(t)= \frac{1}{t} \int_0^t f^*(u) du.
\end{equation*}
By (\ref{sup}), the operator $f\mapsto f^{**}$ is subadditive,
$$
(f+g)^{**}(t)\le f^{**}(t)+g^{**}(t).
$$
Moreover, this operator is bounded in $L^p$ for $p>1,$
\begin{equation*}
||f^{**}||_p\le p'||f||_p,\quad 1<p\le \infty.
\end{equation*}

This estimate follows from Hardy's inequality.

\vskip 8pt

\centerline{\bf Hardy-type inequalities} \vskip 8pt

First, we have the classical Hardy's inequalities (see, e.g.,
\cite[p. 196]{SW}).

\begin{lem}\label{hardy}
Let $\a>0$ and $1\le p<\infty$. Then for any non-negative
measurable on $(0,\infty)$ function $\varphi$,
$$
\biggl(\int_0^{\infty}\Big(\int_0^t\f(u)du\Big)^pt^{-\a-1}dt\biggr)
^{1/p}
\le\frac{p}{\a}\biggl(\int_0^{\infty}\Big(t\f(t)\Big)^pt^{-\a-1}dt
\biggr)^{1/p}
$$
and
$$
\biggl(\int_0^{\infty}\Big(\int_t^{\infty}\f(u)du\Big)^pt^{\a-1}dt\biggr)
^{1/p}
\le\frac{p}{\a}\biggl(\int_0^{\infty}\Big(t\f(t)\Big)^pt^{\a-1}dt
\biggr)^{1/p}.
$$
\end{lem}

\vskip 4pt
 We say that a measurable function $f$ on $(0,\infty)$
is quasi-decreasing if there exists a constant $c>0$ such that
$f(t_1)\le cf(t_2),$ whenever $0<t_2<t_1<\infty$.

We will need also a Hardy-type inequality for quasi-decreasing
functions in the case $0<p<~1$.

\begin{lem}\label{hardy-type}
Let $f$ be a non-negative, quasi-decreasing function on
$(0,\infty).$  Suppose also that $\a>0, \b>-1$ and $0<p<1$. Then
$$
\int_0^{\infty}{u}^{-\a-1}\left(\int_0^{u}f(t)t^{\b}dt\right)^pdu
\le A\int_0^{\infty} u^{-\a-1}\left(f(u)u^{\b+1}\right)^pdu
$$
and
$$
\int_0^{\infty}{u}^{\a-1}\left(\int_0^{u}f(t)t^{\b}dt\right)^p du
\le A\int_0^{\infty} u^{\a-1}\left(f(u)u^{\b+1}\right)^pdu.
$$
\end{lem}

\vskip 4pt Observe that the best constants in these inequalities
were found by Bergh, Burenkov and Persson  \cite{BBP}.

\vskip 6pt

Further, we shall use the following Hardy -- Littlewood
inequality.

\begin{teo} Let $f,g\in S_0(\R^n).$ Then
$$
\int_{\R^n} |f(x)g(x)|dx\le \int_0^\infty f^*(t)g^*(t)dt.
$$
\end{teo}
\begin{proof} Applying Fubini's theorem, we get
$$
\begin{aligned}
\int_{\R^n}
|f(x)g(x)|dx&=\int_{\R^n}dx\int_0^{|f(x)|}\,du\int_0^{|g(x)|}\,dv\\
&=\int_0^\infty \int_0^\infty |\{x: |f(x)|>u, |g(x)|>v\}|dudv\\
&\le \int_0^\infty \int_0^\infty \min(\l_f(u), \l_g(v)) dudv\\
&= \int_0^\infty \int_0^\infty |\{t:f^*(t)>u, g^*(t)>v\}|dudv\\
&=\int_0^\infty f^*(t)g^*(t)dt.
\end{aligned}
$$
\end{proof}
\vskip 6pt

Let $0<p,r<\infty.$ A function $f \in S_0(\mathbb{R}^n)$ belongs
to the Lorentz space $L^{p,r}(\mathbb{R}^n)$ if
\begin{equation*}
\|f\|_{p,r} \equiv \left( \int_0^\infty \left( t^{1/p} f^*(t)
\right)^r \frac{dt}{t} \right)^{1/r} < \infty.
\end{equation*}
For $0<p<\infty,$ the space $L^{p,\infty}(\mathbb{R}^n)$ is
defined as the class of all $f \in S_0(\mathbb{R}^n)$ such that
$$
\|f\|_{p,\infty} \equiv \sup_{t>0}t^{1/p} f^*(t)<\infty.
$$
We have that $||f||_{p,p}=||f||_p.$ Further, for a fixed $p$, the
Lorentz spaces $L^{p,r}$ increase as the secondary index $r$
increases. That is, we have the strict embedding
 $L^{p,r} \subset L^{p,s}$ for $r<s$; in particular,
\begin{equation*}
 L^{p,r} \subset L^{p,p} \equiv L^p, \quad
0<r< p
 \end{equation*}
 (see \cite[p. 217]{BS}).

Let $f\in S_0(\R^n).$ The spherically symmetric rearrangement of
$f$ is defined by
$$
f^*_s(x)=f^*(v_n|x|^n), \quad x\in \R^n,
$$
where $v_n=\pi^{n/2}\Gamma(n/2+1)$ is the measure of the
$n-$dimensional unit ball. The function $f_s^*$ is equimeasurable
with $f,$ it possesses the spherical symmetry and decreases
monotonically as $|x|$ increases.

The classical P\'olya-Szeg\"o principle states that for any $f\in
C_0^\infty(\R^n)$ and any $1\le p\le \infty$
\begin{equation*}
||\nabla f^*_s||_p\le ||\nabla f||_p.
\end{equation*}

Another remarkable result on spherically symmetric rearrangements
is the following theorem.
\begin{teo}\label{rear_ineq}
Let $f,$ $g$, and $h$ be  nonnegative functions in $S_0(\R^n).$
Then
$$
\int_{\R_n}\int_{\R_n}f(y)g(x-y)h(x)dxdy\le
\int_{\R_n}\int_{\R_n}f^*_s(y)g^*_s(x-y)h^*_s(x)dxdy.
$$
\end{teo}

This inequality first was proved for sequences by Hardy and
Littlewood. Afterwards, it was proved for functions by F. Riesz in
the dimension $n=1$ and by Sobolev for any $n\in \N$ (see
\cite{Beckner}, \cite{HLP}, \cite{LL}, \cite{Sob1}).

\vskip 6pt

The basic properties of rearrangements can be found in the books
\cite{BS}, \cite{ChR}, \cite{HLP}, \cite{KPS}, \cite{LL},
\cite{SW}.

\vskip 6pt

\section{Convolutions}

The {\it convolution} of two functions $f$ and $g$ is defined as
\begin{eqnarray}
\label{conv1}
f*g(x)=g*f(x)=\int_{\mathbb{R}^n}f(x-y)g(y)dy.
\end{eqnarray}
If $f\in L^P(\mathbb{R}^n)$ and $g\in L^{p'}(\mathbb{R}^n)$ for
$1\leq p\leq\infty$, then  by H\"{o}lder's inequality the integral
(\ref{conv1}) is absolutely convergent for every
$x\in\mathbb{R}^n$. Furthermore, in this case the convolution is a
continuous function on $\mathbb{R}^n$. At the same time, there are
other conditions that implies almost everywhere convergence of
(\ref{conv1}). For example, if $f,g\in L^1(\mathbb{R}^n),$ then
\begin{eqnarray}
\nonumber
\int_{\mathbb{R}^n}dx\int_{\mathbb{R}^n}|f(x-y)g(y)|dy&=&\int_{\mathbb{R}^n}dy\int_{\mathbb{R}^n}|f(x-y)g(y)|dx\\
\nonumber
&=&\|f\|_1\|g\|_1<\infty
\end{eqnarray}
and whence $f*g\in L^1(\mathbb{R}^n)$. This implies in particular
that $f*g(x)<\infty$ for a.e. $x\in\mathbb{R}^n$.

The following theorem was proved by W.H. Young.
\begin{teo}\label{Young}
Let $1\leq p,q\leq\infty$, $1\leq r\leq\infty$ and
$$
\frac{1}{p}+\frac{1}{q}=1+\frac{1}{r}.
$$
If $f\in L^p(\mathbb{R}^n)$ and $g\in L^q(\mathbb{R}^n),$ then the integral
(\ref{conv1})
converges absolutely for a.e. $x\in\mathbb{R}^n$ and
\begin{eqnarray}
\label{eq2}
\|f\ast g\|_r\leq \|f\|_p\|g\|_q
\end{eqnarray}
\end{teo}
\begin{proof}
We may suppose that $f,g\geq 0$. Let $\varphi=f^p$ and $\psi=g^q$. Observe that $1/p\geq 1/r$ and $1/q\geq 1/r$. We have
\begin{eqnarray}
\nonumber
&f(x-y)g(y)=\varphi(x-y)^{1/p}\psi(y)^{1/q}\\
\nonumber
&=[\varphi(x-y)\psi(y)]^{1/r}\varphi(x-y)^{1/p-1/r}\psi(y)^{1/q-1/r}.
\end{eqnarray}
Let
$$
\frac{1}{\lambda}=\frac{1}{p}-\frac{1}{r},\quad\frac{1}{\mu}=\frac{1}{q}-\frac{1}{r}.
$$
Then
$$
\frac{1}{r}+\frac{1}{\lambda}+\frac{1}{\mu}=1.
$$
Using H\"{o}lder's inequality (with three terms) we have
$$
\begin{aligned}
&\int_{\mathbb{R}^n}f(x-y)g(y)dy\\
&\leq\left(\int_{\mathbb{R}^n}\varphi(x-y)\psi(y)dy\right)^{1/r}\|f\|_p^{1-p/r}\|g\|_q^{1-q/r}.
\end{aligned}
$$
This easily implies (\ref{eq2}).
\end{proof}

Generally, this inequality is not sharp. The best constant in
Young's inequality was found by Beckner \cite{Beckner} in 1975 and
independently by Brascamp and Lieb \cite{BrLieb} in 1976 (see also
\cite{Bar}).

\begin{teo}\label{Optimal_Young} Let $f\in L^p(\R^n)$, $g\in L^q(\R^n)$, $1\le p,q\le \infty$, and $1/p+1/q=1+1/r$
$(r\ge 1).$ Then
\begin{equation}\label{optimal_1}
||f\ast g||_r\le \left(\frac{A_pA_q}{A_r}\right)^n ||f||_p||g||_q,
\end{equation}
where
$$
A_s=[s^{1/s}(s')^{-1/s'}]^{1/2}.
$$
\end{teo}
\vskip 4pt
We shall give some comments.

1. The constant in  (\ref{optimal_1}) is optimal.  It is easy to
check that when $n=1$ and $p,q\not=1,$ there is equality in
(\ref{optimal_1}) for $f(x)=\operatorname{exp}(-p'x^2)$ and
$g(x)=\operatorname{exp}(-q'x^2)$.

\vskip 4pt

2. For the convolution norm we have
$$
\mathcal{C}=\sup \frac{||f\ast g||_r}{||f||_p||g||_q}=\sup
\frac{||f\ast g\ast h||_\infty}{||f||_p||g||_q||h||_{r'}}.
$$

\vskip 4pt

3. The convolution norm for  the dimension $n$ is $C^n$, where $C$
is the one-dimensional norm.

\vskip 4pt

4. Convolution takes radial functions to radial functions (since
it is true for the Fourier transform).

\vskip 4pt

5. The following {\it rearrangement inequality} was used in the
proof of Theorem \ref{Optimal_Young}. This inequality was derived
from the Riesz-Sobolev Theorem \ref{rear_ineq}.

\begin{teo}\label{rear_ineq1}
Let $h,\, f_1,...,f_m$ be  nonnegative functions in $S_0(\R^n).$
Then
$$
\int_{\R_n}h(x)(f_1\ast\cdots\ast f_m)(x)dx\le
\int_{\R_n}h^*_s(x)((f_1)^*_s\ast\cdots\ast (f_m)^*_s(x)dx.
$$
\end{teo}

\vskip 8pt

\centerline{\bf Rearrangement estimate for convolutions}

\vskip 8pt

The following theorem was proved by O'Neil \cite{Neil}.
\begin{teo}\label{ONeil} Let $f,g\in S_0(\mathbb{R}^n)$ be nonnegative functions and $h=f\ast g$. Then for all $t>0$
\begin{eqnarray}
\label{oneil}
h^{**}(t)\leq tf^{**}(t)g^{**}(t)+\int_t^{\infty}f^*(u)g^*(u)du.
\end{eqnarray}
\end{teo}
\begin{proof}
Fix $t>0$ and choose a measurable set $E_t$ of measure $|E_t|=t$ such that
$$
\{x\in\mathbb{R}^n:f(x)>f^*(t)\}\subset E_t\subset\{x\in\mathbb{R}^n:f(x)\geq f^*(t)\}.
$$
Let $f_1(x)=(f(x)-f^*(t))\chi_{E_t}(x)$ and $f_2(x)=f(x)-f_1(x)$. For any measurable set $A\subset\mathbb{R}^n$ of measure $|A|=t$ we have
\begin{eqnarray}
\nonumber
\int_A(f_1*g)(x)dx&=&\int_Adx\int_{\mathbb{R}^n}f_1(x-y)g(y)dy\\
\nonumber
&=&\int_{\mathbb{R}^n}f_1(y)\left(\int_A g(x-y)dx\right)dy\\
\nonumber
&\leq&tg^{**}(t)\int_{E_t}f(y)dy-t^2f^*(t)g^{**}(t)\\
\nonumber
&\leq&t^2f^{**}(t)g^{**}(t)-t^2f^*(t)g^{**}(t).
\end{eqnarray}
Notice that $f_2(x)=f(x)$ if $x\notin E_t$ and $f_2(x)=f^*(t)$ for $x\in E_t$. Clearly $f_2^*(u)=f^*(t)$ for $0<u<t$ and $f_2^*(u)=f^*(u)$ for $u>t$. Thus
\begin{eqnarray}
\nonumber
\int_A(f_2*g)(x)dx&=&\int_Adx\int_{\mathbb{R}^n}f_2(x-y)g(y)dy\\
\nonumber
&\leq&\int_A\left(\int_0^\infty f^*_2(u)g^*(u)du\right)dx\\
\nonumber
&=&\int_A\left(f^*(t)\int_0^t g^*(u)du+\int_t^\infty f^*(u)g^*(u)du\right)dx\\
\nonumber
&=&t^2f^*(t)g^{**}(t)+t\int_0^\infty f^*(u)g^*(u)du.
\end{eqnarray}
It follows that
\begin{eqnarray}
\nonumber
\frac{1}{t}\int_A(f*g)(x)dx&\leq& tf^{**}(t)g^{**}(t)+\int_t^\infty f^*(u)g^*(u)du
\end{eqnarray}
for any measurable set $A$ with measure $|A|=t$, which completes the proof.
\end{proof}

\vskip 8pt

\centerline{\bf Convolution inequalities in Lorentz spaces}

\vskip 8pt

The following theorem was proved by O'Neil \cite{Neil} in 1963
(and later by Hunt \cite{Hunt} in 1966).
\begin{teo}\label{Conv} Let
$$
\frac1p+\frac1q=1+\frac1r\quad (1<p,q,r<\infty)
$$
and
$$
\frac1\mu+\frac1\nu=\frac1s\quad (0<\mu, \nu, s\le\infty).
$$
Then
$$
||f\ast g||_{r,s}\le c||f||_{p,\mu}||g||_{q,\nu}.
$$
\end{teo}

\vskip 4pt

If $s=r,$ we can take $\mu=p(r+1),  \nu=q(r+1)$, and we obtain a stronger inequality, than Young's inequality.

However, the sharp constants in the case of Lorentz norms are unknown.

\vskip 4pt

To prove Theorem \ref{Conv}, it is sufficient to apply the
rearrangement inequality for convolutions and Hardy-type
inequalities.

\vskip 4pt

\vskip 4pt
\section{Fractional integrals}
\vskip 4pt

Let $f$ be a $2\pi-$periodic integrable on $[0,2\pi]$ function with the Fourier series
\begin{equation}\label{FI_1}
f(x)\sim \sum_{n\in \Z}c_n e^{inx}.
\end{equation}
Assume that $c_0=0.$ Set $F(x)=\int_0^x f(t)dt.$ Then $F(0)=F(2\pi)=0$
 and
 $$
 F(x)=\sum_{n\in \Z, n\not=0} \frac{c_n}{in}e^{inx}
 $$
 (the convergence is uniform).

 Let now $\a>0$. Set
 $$
 (in)^\a=|n|^\a\operatorname{exp}\left(\frac12 \pi i\a \operatorname{sign}n\right) \quad (n\not=0)
$$
and consider the series
\begin{equation}\label{FI_2}
\sum_{n\in \Z, n\not=0} \frac{c_n}{(in)^\a}e^{inx}.
\end{equation}
If $\a\in \N,$ then the series (\ref{FI_2}) is obtained by
integration $\a$ times of the series (\ref{FI_1}). It can be shown
(see \cite[Ch. 12]{Z}), that for any $\a>0$ the series
(\ref{FI_2}) converges almost everywhere, it is the Fourier series
of its sum $f_\a$, and satisfies the equality
$$
f_\a(x)=\frac1{\Gamma(\a)}\int_{-\infty}^x f(t)(x-t)^{\a-1}\,dt.
$$
The function $f_\a$ is called {\it the integral of order $\a$} of the function $f$ (Weyl).

\vskip 4pt

The following theorem was proved by Hardy and Littlewood
\cite{HL1}.
\begin{teo}\label{HL_1} Let $f\in L^p[0,2\pi]$ $(1<p<\infty)$. Assume that
$$
\int_0^{2\pi} f(x)dx=0.
$$
Let $0<\a<1/p,\,\,\, p^*=p/(1-\a p).$ Then $f_\a\in L^{p^*}[0,2\pi])$ and
$$
||f_\a||_{p^*}\le c ||f||_p.
$$
\end{teo}

\vskip 4pt

Note that this theorem is not true for $p=1, ~~ 0<\a<1.$ Indeed,
let
$$
f(x)=\frac1x\left(\ln(\frac{\pi}{x})\right)^{\a-2}, ~0< x\le\pi,
~~ f(x)=0\quad\mbox{for}\quad \pi<x\le 2\pi,
$$
and let $f$ be extended periodically  with the period $2\pi$ to
the whole real line. We have $f\in L^1[0,2\pi].$ On the other
hand,
$$
f_\a(x)\ge \frac1{\Gamma(\a)}\int_0^x(x-t)^{\a-1}f(t)dt\ge
\frac{c}{x^{1-\a}}\left(\ln(\frac{\pi}{x})\right)^{\a-1}, ~0<
x\le\pi
$$
$(c>0),$ and thus $f_\a\not\in L^{1/(1-\a)}[0, 2\pi].$

\vskip 4pt

Theorem \ref{HL_1} is one of those results which formed the basis
for the development of the general Embedding Theory of function
spaces. In 1938 Sobolev \cite{Sob1} (see also \cite{Sob2})
extended this theorem to the Riesz potentials of functions of
several variables and obtained embedding theorems  for the spaces
$W_p^r$ he introduced.

\vskip 8pt
\centerline{\bf Riesz potentials}
\vskip 8pt
Let $n\in\N$ and $0<\a<n.$ The Riesz kernel in $\R^n$ is defined by
$$
K_\a(x)=\frac{|x|^{\a-n}}{\g_n(\a)}, \quad \g_n(\a)=\frac{\pi^{n/2}2^\a\Gamma(\a/2)}{\Gamma((n-\a)/2)}.
$$
The Riesz potential of a function $f$ is defined as the convolution
$$
I_\a f(x)=K_\a\ast f(x)=\frac{1}{\g_n(\a)}\int_{\R^n} |x-y|^{\a-n}f(y)dy.
$$
The coefficient $1/\g_n(\a)$ is determined by the equation
$$
\widehat{I_\a f}(\xi)=(2\pi|\xi|)^{-\a}\widehat{f}(\xi), \quad
\xi\in \R^n,  \xi\not= 0.
$$
This equality is understood in the sense of distributions, that is
$$
\int_{R^n}I_\a f(x)\overline {g(x)}
dx=\int_{R^n}(2\pi|\xi|)^{-\a}\widehat{f}(\xi)\overline{\widehat{g}(\xi)}d\xi,
$$
whenever $f$ and $g$ belong to the Schwartz class $S.$

The value of the coefficient $1/\g_n(\a)$ is also important for
the validity of the Riesz composition formula
$$
I_\a\ast I_\b=I_{\a+\b}, \quad \a>0, \,\, \b>0,\,\, \a+\b<n.
$$

Assume that $0<\a<n,~~ 1\le p<n/\a,$ and $f\in L^p(\R^n)$. It is easy to show that the  integral defining  $I_\a f(x)$ converges absolutely
for almost all $x\in \R^n$.

The following Hardy-Littlewood-Sobolev theorem \cite{HL1}, \cite{Sob1} gives an extension
of Theorem \ref{HL_1}.
\begin{teo}\label{HLS} Let $n\in \N,\,\,\,0<\a<n, \,\,\,1<p<n/\a,$ and $p^*=np/(n-\a p)$. If
$f\in L^p(\R^n)$, then
$$
||I_\a f||_{p^*}\le c ||f||_p.
$$
\end{teo}

\vskip 4pt
\begin{rem} The condition
$$
\frac1q=\frac1p-\frac{\a}{n} \quad \left(1<p<\frac{n}{\a}\right)
$$ is not only sufficient, but also necessary for the boundedness
of the operator $f\mapsto I_\a f$ from $L^p$ to $L^q$. It follows
easily by dilation arguments. Indeed, assume that for some $p,
q\in (1,\infty)$
\begin{equation}\label{Bound}
||I_\a f||_q\le A||f||_p
\end{equation}
for any $f\in L^p(\R^n).$ Let $f$ be a characteristic function of
the unit ball in $\R^n$. Set $f_\d(x)=f(\d x)$ for $\d>0.$ Then
$||f_\d||_p=\d^{-n/p}||f||_p.$ Further, $I_\a f_\d(x)= \d^{-\a}
I_\a f(\d x)$ and $||I_\a f_\d||_q=\d^{-\a-n/q} ||I_\a f||_q.$
Thus, by (\ref{Bound}),
$$
\d^{-\a-n/q} ||I_\a f||_q\le A\d^{-n/p}||f||_p.
$$
It follows that $\d^{n/p-n/q-\a}\le C$ for any $\d>0$. It is
possible if and only if $n/p-n/q-\a=0.$
\end{rem}

\vskip 4pt
\begin{rem} Theorem fails to hold for $p=1$ and $p=n/\a.$
\end{rem}
\vskip 4pt
\begin{rem} The classical Sobolev theorem  states that for $n\ge2,$ $1\le p< n$
$$
||f|_{p^*}\le c||\nabla f||_p,\quad p^*=\frac{np}{n-p}.
$$
We stress that, in contrast to Theorem \ref{HLS}, the latter  inequality is true for $p=1$, too.

\end{rem}

\vskip 8pt

\centerline{\bf Hedberg's approach}
\vskip 8pt

Given any locally integrable function $f$ on $\R^n$, denote by $Mf$ the Hardy -- Littlewood maximal function of $f$ defined by
$$
Mf(x)=\sup_{r>0}\frac{1}{|B(x,r)|}\int_{B(x,r)}|f(y)|dy,
$$
where $B(x,r)$ stands for the ball centered at $x$ and having radius $r.$

By the Hardy -- Littlewood maximal theorem,
$$
||Mf||_p\le A_p||f||_p, 1<p\le \infty.
$$
Moreover, $(Mf)^*(t)\asymp f^{**}(t)$ (F. Riesz, C. Herz; see \cite[p. 122]{BS}).

\vskip 4pt

Hedberg \cite{Hed} in 1972 gave a short and elegant proof of
Theorem \ref{HLS} basing on the following lemma.

\begin{lem}\label{Hedberg_1}
If $0<\a<n,$ then for all $x\in\R^n$ and any $\d>0,$
$$
\int_{|x-y|\le \d} |f(y)||x-y|^{\a-n}\,dy\le A\d^\a Mf(x).
$$
\end{lem}
\begin{proof} We have
$$
\begin{aligned}
&\int_{|x-y|\le \d} |f(y)||x-y|^{\a-n}\,dy\\
&=\sum_{k=0}^\infty \int_{2^{-k-1}\d<|x-y|\le
2^{-k}\d}|f(y)||x-y|^{\a-n}\,dy\\
&\le A\sum_{k=0}^\infty (2^{-k}\d)^{\a-n}\int_{|x-y|\le
2^{-k}\d}|f(y)|\,dy\\
&\le A\d^\a Mf(x)\sum_{k=0}^\infty 2^{-k\a}=A'\d^\a Mf(x).
\end{aligned}
$$
\end{proof}

\vskip 4pt
\begin{prop}\label{Hedberg_2} For any $0<\a<n,$ $1\le p<n/\a$, and
any measurable function $f\ge 0$ for all $x\in\R^n$
\begin{equation}\label{hedberg}
I_\a f(x)\le A||f||_p^{\a p/n}Mf(x)^{1-\a p/n}.
\end{equation}
\end{prop}
\begin{proof} Let $\d>0$. By H\"older's inequality,
$$
\int_{|x-y|\ge \d} |f(y)||x-y|^{\a-n}\,dy\le A\d^{\a-n/p}||f||_p.
$$
By Lemma \ref{Hedberg_1},
$$
\int_{|x-y|\le \d} |f(y)||x-y|^{\a-n}\,dy\le A\d^\a Mf(x).
$$
Choosing
$$
\d = \left(\frac{||f||_p}{Mf(x)}\right)^{p/n},
$$
we obtain (\ref{hedberg}).
\end{proof}

\vskip 4pt {\bf Proof of Theorem \ref{HLS}.} Let $1/p-1/q=\a/n.$
Then by (\ref{hedberg})
$$
(I_\a f(x))^q\le A^q||f||_p^{\a pq/n}Mf(x)^p.
$$
Thus,
$$
||I_\a f||_q\le A||f||_p^{\a p/n}||Mf||_p^{p/q}.
$$
Since $p/q=1-\a p/n$ and $||Mf||_p\le A_p||f||_p,$ we obtain that
$$||I_\a f||_q\le A'||f||_p.$$

\vskip 8pt
\centerline{\bf Fractional maximal functions and Riesz potentials}
\vskip 8pt

For $0<\a<n,$ set
$$
M_\a f(x)=\sup_{r>0}\frac{1}{|B(x,r)|^{1-\a/n}}\int_{B(x,r)}|f(y)|dy.
$$

If $f\ge 0$, then $I_\a f$ is estimated pointwise by the fractional maximal function from below. Indeed,
$$
\begin{aligned}
\int_{\R^n} \frac{f(y)}{|x-y|^{n-\a}}dy&\ge \int_{B(x,r)}\frac{f(y)}{|x-y|^{n-\a}}dy\\
&\ge \frac1{r^{n-\a}}\int_{B(x,r)}|f(y)|dy
\end{aligned}
$$
for any $r>0$ and any $x\in\R^n.$ Thus,
$$
I_\a f(x)\ge c M_\a f(x) \quad (c>0).
$$

The opposite inequality is in general false (for example, take $f(y)=|y|^{-\a}$ and $x=0$).

Nevertheless, Muckenhoupt and Wheeden \cite{MuWh} (see also
\cite{AdHe})  proved the following theorem.
\begin{teo}
 Lat $1<p<\infty$ and $0<\a<n$. Then
$$
||I_\a f||_p\le c ||M_\a f||_p.
$$
\end{teo}

Actually this inequality was proved for $L^p_w,$ where $w$ is an $A_\infty$ weight.

\vskip 8pt

\centerline{\bf Estimates of rearrangements} \vskip 8pt

 As it was mentioned above, for the classical Hardy -- Littlewood maximal function we have
 (F. Riesz, C. Herz)
 $$
c^{-1}f^{**}(t)\le (Mf)^*(t)\le c f^{**}(t)\quad (c>0).
$$

Observe that for $\f(x)=1/|x|$ $(x\in \R^n))$ we have
$\f^*(t)=(v_n/t)^{1/n}$, where $v_n$ is the measure of the
$n-$dimensional unit ball. Indeed,
$$
|\{x: \frac1{|x|}>y\}|=\frac{v_n}{y^n}\quad(y>0),
$$
and we find $y=\f^*(t)$ from the equation $t=v_n y^{-n}.$  Thus,
for the Riesz kernel $K_\a(x)=|x|^{\a-n}/\g_n(\a)$ $(0<\a<n)$ we
have
$$
K_n^*(t)=\frac{c}{t^{1-\a/n}}\quad\mbox{and}\quad K_n^{**}(t)=\frac{c'}{t^{1-\a/n}}
$$
Applying O'Neil's inequality, we obtain $(0<\a<n)$
$$
(I_\a f)^{**}(t)\le c\left[t^{\a/n}f^{**}(t)+ \int_t^\infty s^{\a/n-1}f^*(s)ds\right].
$$

It follows that the operator $I_\a$ is bounded from $L^1$ into $L^{n/(n-\a),\infty}$ and from $L^{n/\a,1}$ into $L^\infty$.
These statements are sharp.

Cianchi, Kerman, Opic, and Pick \cite{CKOP}  obtained a
sharp rearrangement inequality for $M_\a f.$
\begin{teo}\label{CKOP}
 Let $n\in \N$ and $0<\a<n.$ Then there exists a  constant $C>0$ , depending only on $n$ and $\a$, such that
 $$
 (M_\a f)^*(t)\le C \sup_{t\le s<\infty} s^{\a/n}f^{**}(s), \,\, t>0,
 $$
 for every $f\in L^1_{\operatorname{loc}}(\R^n).$
 \end{teo}
\begin{proof}
Denote
$$
\Phi(x,r)=|B(x,r)|^{\alpha/n-1}\int_{B(x,r)}|f(y)|dy
$$
and
$$
F(x)=\sup_{r>0}\Phi(x,r).
$$
We have that
$$
\Phi(x,r)\le [v_n^{1/n}r]^\alpha\min[Mf(x), f^{**}(v_nr^n)].
$$
Fix $t>0.$ Set
$$
F_1(x)=\sup_{0<r<t^{1/n}}\Phi(x,r),~~ F_2(x)=\sup_{r\ge
t^{1/n}}\Phi(x,r).
$$
We have
$$
F_1(x)\le v_n^{\alpha/n}t^{\alpha/n}Mf(x).
$$
Thus,
$$
F_1^*(t)\le c t^{\alpha/n}f^{**}(t).
$$
On the other hand,
$$
F_2(x)\le \sup_{r\ge t^{1/n}}[v_n^{1/n}r]^\alpha  f^{**}(v_n r^n).
$$
It follows that
$$
||F_2||_\infty\le c\sup_{s\ge t}s^{\alpha/n}f^{**}(s).
$$
\end{proof}
\vskip 6pt

If $\a=0,$ we get the Riesz estimate $(Mf)^*(t)\le C f^{**}(t).$

\vskip 6pt

\begin{cor} The operator $M_\a$ is bounded from $L^{n/\a,\infty}$ to $L^\infty.$
\end{cor}

\vskip 8pt

\centerline{\bf Refiniment of the Hardy-Littlewood-Sobolev
Theorem}

 \vskip 8pt

 The following theorem was obtained by O'Neil \cite{Neil}.
 \begin{teo}\label{ONeil_2} Let $n\in \N,\,\,\,0<\a<n, \,\,\,1<p<n/\a,$ and $p^*=np/(n-\a p)$. If
$f\in L^p(\R^n)$, then $I_\a f\in L^{p^*,p}/\R^n)$, and
$$
||I_\a f||_{p^*,p}\le c ||f||_p.
$$
\end{teo}

Indeed, we have seen that, by O'Neil inequality,
$$
(I_\a f)^{**}(t)\le c\left[t^{\a/n}f^{**}(t)+ \int_t^\infty
s^{\a/n-1}f^*(s)ds\right].
$$
It remains to apply Hardy's inequalities.

\vskip 4pt

Since we have a strict embedding $L^{p^*,p}(\R^n)\subset
L^{p^*}(\R^n)$, Theorem \ref{ONeil_2} provides a refinement of the
Hardy-Littlewood-Sobolev Theorem.

\vskip 8pt

\centerline{\bf The limiting case $p=n/\a, \,\, 0<\a<n$}

\vskip 8pt The following theorem was proved by D. Adams
\cite{Adams}.
\begin{teo}\label{Adams} Let $n\in \N,\,\,\,0<\a<n, \,\,\,p=n/\a,$ and let $f\in L^p(\R^n).$
Assume that $\operatorname{supp} f\subset B(0,R)$ and that
$||f||_p=1.$ Then there is a constant $A=A(n,\a)$ such that
$$
\int_{B(0,R)}\operatorname{exp}\left(\b_0|I_\a
f(x)|^{p'}\right)dx\le AR^n,
$$
where $\b_0=\b_0(n,\a)=\g_n(\a)n/\omega_{n-1}$ ($\omega_{n-1}$ is
the area of the $n-$dimensional unit sphere).
\end{teo}
\vskip 6pt This theorem has an interesting history. Moser
\cite{Moser} in 1971 proved that
\begin{equation}\label{Moser}
\int_{B(0,R)}\operatorname{exp}\left(\b|u(x)|^{n'}\right)dx\le
AR^n
\end{equation}
for all $0<\b\le \b_0=n\omega_{n-1}^{1/(n-1)}$ and all $u\in \stackrel{\circ}{W}_n^1(B(0,R))$ with
$$
 \int_{B(0,R)}|\nabla u(x)|^n\,dx\le 1
 $$
    ($\stackrel{\circ}{W}_n^1(\Omega)$ is the closure of
    $C_0^\infty(\Omega)$ in    $W_n^1(\Omega)).$ It was known before
    that (\ref{Moser}) holds for {\it some} $\b>0.$

    \vskip 6pt

    Theorem of Adams implies Moser's theorem.   Indeed, the
    identity
    $$
    g(x)=\frac1{\omega_{n-1}} \int_{\R^n} \frac{\nabla g(y)\cdot
    (x-y)}{|x-y|^n}\,dy
    $$
    yields that
    $$
    |g(x)|\le \frac{\g_n(1)}{\omega_{n-1}}I_1(|\nabla g|)(x).
    $$

    Observe that the opposite isn't true.
    \vskip 6pt

Finishing this section, we mention the paper by Eiichi Nakai
\cite{Nakai}  which contains a short but comprehensive survey of
the results on Riesz potentials.

 \vskip 8pt

    \section{Bessel potentials}
    \vskip 8pt

    The Riesz potentials have a lot of important applications.
    However, the kernels $K_\a(x)$ go slowly to zero as $|x|$
    tends to $\infty.$ This creates some difficulties in the use
    of Riesz potentials. In particular, the operator $I_\a$ is not
    bounded on $L^p.$

    It is natural to replace $K_\a$ with a function that has the
    same singularity at $0$   but a more rapid decay at $\infty.$

    The Fourier transform of a function $f\in L^1(\R^n)$ is
    defined by
    $$
    \widehat{f}(\xi)= \int_{\R^n} f(x)e^{-i2\pi x\cdot \xi}dx,
    \quad \xi\in \R^n.
    $$

    The Bessel kernel of order $\a>0$ on $\R^n$ is defined as the
function for which the Fourier transform equals
$$
\widehat G_\a(\xi)=(1+4\pi^2|\xi|^2)^{-\a/2}, \quad \xi\in \R^n.
$$
The following equality holds (see \cite[8.1]{Nik})
$$
G_\a(x)=\frac{1}{(4\pi)^{\a/2}\Gamma(\a/2)}\int_0^\infty
z^{(\a-n)/2-1}e^{-\pi |x|^2/z-z/(4\pi)}\,dz.
$$

We have
$$
\int_{\R^n} G_\a(x)dx=1,
$$
$$
G_{\a+\b}=G_\a\ast G_\b\quad (\a,  \b>0).
$$
Further, for $0<\a<n,$
$$
G_\a(x)=K_\a(x)+o(|x|^{\a-n}),\quad x\rightarrow 0.
$$
Thus, the local behaviour of the kernel $G_\a(x)$ as $x\rightarrow
0$ is the same as that of $K_\a$. The advantage of the Bessel
kernels is that they decrease sufficiently fast on the infinity,
$$
G_\a(x)=O(e^{-|x|/2}),\quad  x\rightarrow \infty.
$$

If $f\in L^p(\R^n)\,\,\,(1\le p\le \infty),$ then the Bessel
potential of order $\a>0$ for the function $f$ is defined by the
equality
$$
J_\a f(x)=G_\a \ast f(x).
$$
By Minkowski's inequality,
$$
||J_\a f||_p\le ||G_\a||_1||f||_p=||f||_p,\quad 1\le p\le \infty.
$$

If $\a\ge n,$ then $G_\a\in L^p(\R^n)$ for any $1\le p\le \infty.$
In the case $0<\a<n$ we have $G_\a\in L^s(\R^n)$ for all $1\le
s<n/(n-\a).$ Applying Young's inequality, we have
$$
||J_\a f||_q\le c ||f||_p, \quad 1\le p\le q\le \infty, \,\,
1/p-1/q<\a/n.
$$

By the Hardy -- Littlewood -- Sobolev theorem on Riesz potentials,
$$
||I_\a f||_{p^*}\le c||f||_p, \quad 1<p<\frac{n}{\a},\,\,
0<\a<n,\,\, p^*=\frac{np}{n-\a p}.
$$
Since $G_\a$ is majorized by $K_\a$ on the unit ball and
$G_\a(x)=O(e^{-|x|/2})$ at infinity, we obtain that
$$
||J_\a f||_{p^*} \le c||f||_p, \quad 1<p<\frac{n}{\a},\,\, 0<\a<n.
$$
\vskip 8pt

\centerline{\bf Relations between Bessel and Riesz potentials}
\vskip 8pt

Let $\a>0$ and $1<p<\infty.$ Let $F=G_\a\ast f,$ where $f\in
L^p(\R^n)$. Then $F$ can be represented as $F=K_\a\ast g,$ where
$g\in L^p(\R^n)$.

Indeed,
$$
\widehat{f}(\xi)=(1+4\pi^2|\xi|^2)^{\a/2}\widehat{F}(\xi).
$$
From here,
\begin{equation}\label{Bessel_10}
\left(\frac{4\pi^2|\xi|^2}{1+4\pi^2|\xi|^2}\right)^{\a/2}\widehat{f}(\xi)=(2\pi|\xi|)^{\a/2}\widehat{F}(\xi).
\end{equation}
We  show that the left-hand side of (\ref{Bessel_10}) is the
Fourier transform of some function $g\in L^p(\R^n)$.

We shall use the binomial expansion
$$
(1-t)^{\a/2}=1+\sum_{m=1}^\infty A_{m,\a} t^m, \quad |t|<1,
$$
$$
A_{m,\a}=(-1)^m\frac{\a(\a-1)\cdots(\a-m+1)}{m!}.
$$
All the $A_{m,\a}$ have the same signs for $m$ sufficiently large.
Thus,
$$
\sum_{m=1}^\infty |A_{m,\a}|<\infty.
$$
Taking $t=(1+4\pi^2|\xi|^2)^{-1}$, we get
$$
\begin{aligned}
\left(\frac{4\pi^2|\xi|^2}{1+4\pi^2|\xi|^2}\right)^{\a/2}&=1+\sum_{m=1}^\infty
A_{m,\a}(1+4\pi^2|\xi|^2)^{-m}\\
&=1+\sum_{m=1}^\infty A_{m,\a}\widehat{G_{2m}}(\xi).
\end{aligned}
$$
We observe that
$$
G_\b(x)\ge 0 \quad\mbox{and}\quad \int_{R^n}G_\b(x)dx=1, \,\,\,
\b>0.
$$
Setting
$$
h(x)= \sum_{m=1}^\infty A_{m,\a}G_{2m}(x),
$$
we have that $h\in L^1(\R^n)$ and
$$
\left(\frac{4\pi^2|\xi|^2}{1+4\pi^2|\xi|^2}\right)^{\a/2}\widehat{f}(\xi)=\widehat{f}(\xi)+\widehat{f\ast
h}(\xi).
$$
We see that the left-hand side of (\ref{Bessel_10}) is the Fourier
transform of the function $g=f+f\ast h.$ Obviously, $g\in
L^p(\R^n).$ Finally, (\ref{Bessel_10}) yields that $F=K_\a\ast g.$

Clearly, the converse is not true - a function represented as a
Riesz potential may not be represented as a Bessel potential.
Indeed, the operator $I_\a$ is not bounded in $L^p.$

\newpage

  \centerline{\bf Fractional Sobolev
spaces}

\vskip 8pt

Let $1\le p\le\infty$ and $r\in \N.$ Denote by $W_p^r(\R^n)$ the
Sobolev space of functions $f\in L^p(\R^n)$ for which all weak
derivatives $D^sf$ $(s=(s_1,...,s_n))$ of order
$|s|=s_1+\cdots+s_n\le r$ exist and belong to $L^p(\R^n).$ The
norm in $W_p^r$ is defined by
$$
||f||_{W_p^r}=\sum_{|s|\le r}||D^sf||_p.
$$

\vskip 4pt

A natural extension of the Sobolev spaces to fractional values of $r$ give the spaces of Bessel potentials
(Sobolev -- Liouville spaces; Aronszajn and Smith; Calder\'on, 1961).

Let $1\le p\le \infty$ and $\a>0.$ We say that a measurable on
$\R^n$ function $f$ belongs to the space $L_p^\a(\R^n)$,
 if there exists a function $g\in L^p(\R^n)$ such that $f(x)=G_\a\ast g (x)$ for almost all $x\in \R^n$  (this function is unique). The norm in
$L^\a_p(\R^n)$  is defined by $||f||_{L^\a_p}=||g||_p.$ We observe
that
$$
L_p^\a(\R^n)\subset L^p(\R^n),\quad L_p^\a(\R^n)\subset L_p^\b(\R^n)\quad (\a>\b, \,\, 1\le p\le \infty).
$$

The following theorem was proved independently by A. Calder\'on \cite{Cald} and Lizorkin \cite{Liz1}.

\begin{teo}\label{C-L} Let $r\in \N$ and $1<p <\infty.$ Then
$$
W_p^r(\R^n)=L_p^r(\R^n),
$$
and the norms are equivalent.
\end{teo}

\vskip 4pt

The relations between $W_p^r(\R^n)$ and $L_p^r(\R^n)$ in the extreme cases $p=1$ and $p=\infty$ are the following.
\vskip 4pt

(i) When $n=1$, then $W_p^r(\R)=L_p^r(\R)$, if $r$ is even and $p=1,\,p=\infty.$

\vskip 4pt

(ii) When $n\ge 2,$ then $W_p^r(\R^n)\subset L_p^r(\R^n)$, if  $r$ is even and $p=1,\,p=\infty;$
the reverse inclusion fails for both $p=1$ and $p=\infty.$

\vskip 4pt

(iii) For all $n$, if $r$ is odd, then no one of the spaces $W_p^r(\R)$ and $L_p^r(\R)$ is contained in the other.

\vskip 6pt

Observe that the proof of (i) is quite simple. Indeed,
$$
\widehat{f''}(\xi)=-4\pi^2\xi^2\widehat f(\xi).
$$
On the other hand,
the equality $f=G_2\ast g$ is equivalent to
$$
\widehat g(\xi)=(1+4\pi^2\xi^2)\widehat f(\xi).
$$
Hence, $\widehat g(\xi)=\widehat f(\xi)-\widehat{f''}(\xi)$ and $g=f-f''.$ Applying induction on $r,$ we obtain (i).

\vskip 6pt

Also, we shall briefly describe the relations between fractional
Sobolev spaces $L_p^\a$ and other spaces of fractional smoothness - Besov spaces.

\vskip 4pt

Let a function $f$ be given on $\R^n.$ For $r\in \mathbb N$  and $h\in \mathbb{R}^n$ we set
$$
\Delta^r(h)f(x)=\sum_{i=0}^r (-1)^{r-i}{r\choose i} f(x+ih).
$$
 If
$f\in L^{p}(\mathbb{R}^n),$ then the function
$$
\o^r(f;\delta)_{p}=\sup_{0\le |h|\le \delta}||\Delta^r(h)f||_{p}
$$
is called the {\it  modulus of continuity of order $r$} of the
function $f$  in $L^{p}.$ \vskip 6pt Let $\a>0$ and $1\le
p,q<\infty.$ Assume that  $r>\a,\,\, r\in \N.$ A function $f\in
L^p(\R^n)$ belongs to the class $B_{p,q}^\a(\R^n)$ if
$$
||f||_{B_{p,q}^\a}\equiv
||f||_p+\left(\int_0^{\infty}\left(t^{-\a}\o^r(f;t)_p\right)^q
\frac{dt}{t}\right)^{1/q}<\infty.
$$

The Nikol'skii space $H_p^\a(\R^n)\equiv B_{p,\infty}^\a(\R^n)$ is defined as the class of all functions $f\in
L^p(\R^n)$ such that
$$
\o^r(f;t)_p=O(\d^\a),\quad r>\a.
$$

 Denote also $B^\a_{p,p}\equiv B^\a_{p}.$

 \vskip 4pt

 The following relations hold.

 $$
 B_{p,q}^\a(\R^n)\subset L_p^\a(\R^n), \quad 1\le q\le\min(p,2)
 $$
 and
 $$
L_p^\a(\R^n)\subset  B_{p,q}^\a(\R^n), \quad q\ge \max(p,2).
$$
\vskip 4pt
If $1\le p\le 2,$ then
$$
B_p^\a\subset L_p^\a\subset B_{p,2}^\a.
$$
If $2\le p<\infty$, then
$$
B_{p,2}^\a\subset  L_p^\a\subset B_{p}^\a.
$$
For $1\le p\le \infty$
$$
B_{p,1}^\a\subset L_p^\a\subset B_{p,\infty}^\a.
$$

\vskip 6pt

\section{Anisotropic spaces}

\vskip 6pt

Let $r\in \N,~~1\le p<\infty,$ and $1\le j\le n.$ Denote by
$W_{p;j}^r(\R^n)$ the Sobolev space of all functions $f\in
L^p(\R^n)$ for which there exists the weak partial derivative
$D_j^rf\in L^p(\R^n).$ Set also
$$
W_p^{r_1,...,r_n}(\R^n)=\bigcap_{j=1}^nW_{p;j}^{r_j}(\R^n)\quad(r_j\in\N,~~1\le
p<\infty).
$$
The norm in $W_p^{r_1,...,r_n}(\R^n)$ is defined by
$$
||f||_{W_p^{r_1,...,r_n}}= ||f||_p +  \sum_{j=1}^n||D_j^rf||_p.
$$

We observe that
$$
W_p^{r,...,r}(\R^n)= W_p^r(\R^n),\quad 1<p<\infty.
$$
It is a consequence of the following theorem (K. Smith
\cite{Smith}).
\begin{teo}\label{Smith}  Let $1<p<\infty$ and $r\in \N.$ Then for
any multi-index $s=(s_1,...,s_n)$ with non-negative integer
components such that
$$
|s|=\sum_{i=1}^n s_i= r,
$$
the weak derivative $D^sf$ exists and
$$
||D^sf||_p\le c \sum_{i=1}^n||D_i^rf||_p.
$$
\end{teo}

For $p=1$ and $p=\infty$ this theorem fails. Namely, Ornstein
\cite{Orn} constructed a function $f$ on $\R^2$ such that
$||D^{1,1}f||_1$ cannot be estimated by
$C(||D^2_1f||_1+||D^2_2f||_1).$ For $p=\infty$ a counterexample is
given by the function which equals to $xy\ln|\ln (x^2+y^2)|$ in
some neighborhood of the origin (see \cite[Ch. 3]{BIN}).

We will also  consider the fractional Sobolev spaces. These spaces were introduced and studied in the sixties by Lizorkin.

Recall that the Bessel kernel of order $\a>0$ in $\R$ is defined as the
function with Fourier transform
$$
\widehat G_\a(\xi)=(1+4\pi^2\xi^2)^{-\a/2}, \quad \xi\in \R.
$$

 Let $1\le p \le \infty,$  $\a>0,$ and $1\le j\le n.$ Let $f$
be a measurable function on $\R^n.$ We say that $f$ belongs to the
space $L_{p;j}^\a(\R^n)$ if there exists a function $f_j\in
L^p(\R^n)$ such that for almost all $x\in \R^n$
\begin{equation}\label{conv}
f(x)=\int_\R G_\a(x_j-t)f_j(t,\widehat x_j)\,dt.
\end{equation}
It is not difficult to prove that the equality (\ref{conv}) determines the
function $f_j$ uniquely, up to its values on a set of
$n-$dimensional Lebesgue measure zero. We have
$$
||f||_p\le ||f_j||_p.
$$
We call $f_j$ the {\it {Bessel derivative}} of the function $f$ of
order $\a$ with respect to $x_j$. We denote it by $J_j^\a f.$
If $\a$ is fractional, we use also the standard notation $D_j^\a f.$
However, for integer $\a$ we keep the latter notation for the usual weak derivative.

The following equality holds
$$
L_{p;j}^\a(\R^n)=W_{p;j}^\a(\R^n)\quad(1<p<\infty,\,\, \a\in \N)
$$
and the norms are equivalent.

If $1\le p\le\infty$ and $\a>0,$ we set $\tilde L_{p,j}^\a =
L_{p,j}^\a$ if $\a$ is fractional, and $\tilde L_{p,j}^\a =
W_{p,j}^\a$ if $\a$ is integer.

Let $\a_j>0~~(j=1,...,n)$ and $1\le p\le \infty.$ Set
$$
L_p^{\a_1,...,\a_n}(\R^n)=\bigcap_{j=1}^n\tilde L_{p;j}^{\a_j}(\R^n)
$$
and
$$
||f||_{L_p^{\a_1,...,\a_n}}=\sum_{j=1}^n ||f||_{\tilde L_{p;j}^{\a_j}}
$$
We shall call $L_p^{\a_1,...,\a_n}(\R^n)$ a fractional Sobolev
space or a Sobolev-Liouville space. For integer $\a_j$ and $1\le p\le\infty$
$$
W_p^{\a_1,...,\a_n}(\R^n)=L_p^{\a_1,...,\a_n}(\R^n).
$$

We observe that Lizorkin defined $L_p^{\a_1,...,\a_n}(\R^n)$ as the intersection
$$
\bigcap_{j=1}^n L_{p;j}^{\a_j}(\R^n)
$$
(that is, he used only Bessel's derivatives). Our definition
differs from this only in the case when $p=1$ or $p=\infty$ and at
least one of the $\a_j$ is an odd integer.

We have the equality
$$
L_p^{\a,...,\a}(\R^n)=L_p^\a(\R^n)\quad (\a>0,\,\, 1<p<\infty),
$$
and the norms are equaivalent (Lizorkin \cite{Liz},
Strichartz \cite{Str}).

It is possible to show that $L_1^\a(\R^n)\not\subset L_1^{\a,...,\a}(\R^n)$ for $n\ge 2$. We have an {\it open problem}: is it true that
$$
 L_1^{\a,...,\a}(\R^n)\subset L_1^\a(\R^n),\quad n\ge 2 \,?
 $$
 For $W-$ spaces it is obviously true.

 \vskip 8pt

 \centerline{\bf Embeddings}
 \vskip 8pt

Using O'Neil's inequality for  Riesz potentials, we have for $1<p<n/\a,\,\,\, p^*=np/(n-\a p)$
$$
L_p^\a(\R^n)\subset L^{p^*,p}(\R^n).
$$
We stress that this embedding fails for $p=1.$

In the anisotropic case Lizorkin \cite{Liz} proved the following
Sobolev type embedding.
\begin{teo}\label{Lizor} Let $\a_j>0 ~~(j=1,...,n)$,
$$
\a=n\left(\sum_{j=1}^n\frac1{\a_j}\right)^{-1}, \,\, 1<p<\frac
{n}{\a}, \quad\mbox{and}\quad p^*=\frac{np}{n-\a p}.
$$
Then for every function $f\in
L_p^{\a_1,...,\a_n}(\R^n)$
$$
||f||_{p^*} \le c||f||_{L_p^{\a_1,...,\a_n}}.
$$
\end{teo}

\vskip 6pt

Applying estimates of rearrangements, we proved \cite{K1998}
 a refinement of this theorem.

\begin{teo}\label{Kol1}  Assume that $1<p<\infty,~~n\ge
1$ or $p=1,~~n\ge 2.$ Let $\a_j>0 ~~(j=1,...,n)$ and let
$$
\a\equiv n\left(\sum_{j=1}^n\frac1{\a_j}\right)^{-1}<\frac np .
$$  Set
$p^*=np/(n-\a p).$ Then for every function $f\in
L_p^{\a_1,...,\a_n}(\R^n)$
\begin{equation}\label{emb1}
||f||_{p^*,p}\le c \sum_{j=1}^n||D_j^{\a_j}f||_p.
\end{equation}
\end{teo}

Emphasize that, in contrast to the  case $n=1$,  for $n\ge 2$ Theorem  \ref{Kol1} is true for $p=1,$
too.

Since $||f||_{p^*}\le c||f||_{p^*,p},$ Theorem \ref{Kol1} provides
a refinement of Theorem \ref{Lizor}. We see also that for Lizorkin
spaces Theorem \ref{Lizor}
 holds in the case when $p=1$, $n\ge 2,$ and all $\a_j$ are non-integer.

Let $\a_1=\cdots=\a_n=\a.$ If $1<p<\infty,$ then
$L_p^{\a,...,\a}=L_p^\a,$ and we obtain embedding
$$
L_p^\a(\R^n)\subset L^{p^*,p}(\R^n)\quad \left(1<p<\frac
n\a,\,\,p^*=\frac{np}{n-\a p}\right)
$$
mentioned above. For $p=1$ it doesn't hold. However, Theorem \ref{Kol1} holds for $p=1,\,\, n\ge 2,$ too.
That is, we have
$$
L_1^{\a,...\a}(\R^n)\subset L^{n/(n-\a),1}(\R^n),\quad 0<\a<n, \,\, n\ge 2.
$$

We proved also estimates of difference norms (embeddings to the
Besov spaces).

Let a function $f$ be given on $\R^n.$ We have already defined the
modulus of continuity $\o^r(f;\delta)_{p}$. Namely, for $r\in
\mathbb N$  and $h\in \mathbb{R}^n$ we set
$$
\Delta^r(h)f(x)=\sum_{i=0}^r (-1)^{r-i}{r\choose i} f(x+ih),
$$
 If
$f\in L^{p}(\mathbb{R}^n),$ then the function
$$
\o^r(f;\delta)_{p}=\sup_{0\le |h|\le \delta}||\Delta^r(h)f||_{p}
$$
is called the {\it  modulus of continuity of order $r$} of the
function $f$  in $L^{p}.$

Now we define the partial moduli of continuity.  Let $r\in \mathbb
N,$ $1\le j\le n,$ and $h\in \mathbb{R}$. Set
$$
\Delta_j^r(h)f(x)=\sum_{i=0}^r (-1)^{r-i}{r\choose i} f(x+ihe_j),
$$
where $e_j$ is the unit coordinate vector in $\mathbb{R}^n.$ If
$f\in L^{p}(\mathbb{R}^n),$ then the function
$$
\o_j^r(f;\delta)_{p}=\sup_{0\le h\le \delta}||\Delta_j^r(h)f||_{p}
$$
is called the {\it partial modulus of continuity} of order $r$ of
the function $f$ with respect to the variable $x_j$ in $L^{p}.$ If
$r=1,$ then we omit the superscript in this notation.

If $f$ has the weak derivative $D_j^rf\in
L^1_{loc}(\mathbb{R}^n),$ then
\begin{equation*}
\Delta_j^r(h)f(x)=\int_0^h\cdots\int_0^hD_j^rf(x+(u_1+...+u_r)e_j)du_1...du_r
\end{equation*}
for almost all $x$.

For $r>0,$ let $\bar r$ be the least integer such that $r\le
\bar{r}.$ We have the following theorem (see \cite{K1988},
\cite{K1993}, \cite{K2001}).

\begin{teo} Let $r_1,...,r_n$ be positive numbers and let
$$
r=n\left(\sum_{j=1}^n\frac{1}{r_j}\right)^{-1}, \,\, 1\le
p<q<\infty, \,\, \varkappa=1-\frac nr\left(\frac 1p-\frac
1q\right)>0.
$$
Set $\a_j=\varkappa r_j \,\,(j=1,...,n).$ If $1<p<\infty$ and
$n\ge 1$, or $p=1$ and $n\ge 2,$ then for  every function $f\in
L^{r_1,...,r_n}_p(\R^n)$
$$
\sum_{j=1}^n\left(\int_0^\infty
[h^{-\a_j}\o_j^{\bar{r}_j}(f;h)_q]^p\frac{dh}{h}\right)^{1/p}\le c
\sum_{j=1}^n||D_j^{r_j}f||_p.
$$
\end{teo}

\vskip 8pt \centerline{\bf Nikolskii-Besov and Lipschitz spaces}

\vskip 8pt

 Let $f\in L^p(\R^n) ~~(1\le p\le
\infty),~~\a>0,$ and $1\le j\le n.$ Let $r$ be the least integer
such that $r>\a.$ The function $f$ belongs to the class
$H_{p;j}^\a(\R^n)$ if
\begin{equation}\label{nik}
||f||_{H_{p;j}^\a}\equiv
||f||_p+\sup_{\d>0}\frac{\o_j^r(f;\d)_p}{\d^{\a}}<\infty.
\end{equation}
Emphasize that if $\a\in \N,$ then in (\ref{nik}) we take the
modulus of continuity of the order $r=\a+1.$

If $\a_j>0 ~~(j=1,...,n)$ and $1\le p\le \infty,$ the Nikol'skii
space $H_p^{\a_1,...,a_n}(\R^n)$ is defined by
$$
H_p^{\a_1,...,a_n}(\R^n)=\bigcap_{j=1}^n H_{p;j}^{\a_j}(\R^n).
$$

Assume now that $\a>0,$ $1\le p,q<\infty,$ and $1\le j\le n.$ As
above, let $r$ be the least integer such that $r>\a.$ A function
$f\in L^p(\R^n)$ belongs to the class $B_{p,q;j}^\a(\R^n)$ if
\begin{equation*}
||f||_{B_{p,q;j}^\a}\equiv||f||_p+\left(\int_0^{\infty}\left(t^{-\a}\o_k^r(f;t)_p\right)^q
\frac{dt}{t}\right)^{1/q}<\infty.
\end{equation*}
 Denote also $B^\a_{p,p;j}\equiv B^\a_{p;j}.$

Let $\a_j>0 ~~(j=1,...,n)$ and $1\le p,q<\infty.$ Then we set
 $$
 B^{\a_1,...,\a_n}_{p,q}({\mathbb R}^n)
=\bigcap_{j=1}^nB^{\a_j}_{p,q;j}({\mathbb R}^n)\quad
(B^{\a_1,...,\a_n}_{p}\equiv B^{\a_1,...,\a_n}_{p,p}).
$$

It is easy to see that
$$
||f||_{H_{p;j}^\a}=\lim_{\t\to +\infty}\|f\|_{B^\a_{p,\t;j}}.
$$
This is why we set, by definition,
$B_{p,\infty;j}^\a(\R^n)=H_{p;j}^\a(\R^n).$

It is also well known that
$$
B^\a_{p,\t;j}\subset B^\a_{p,\eta;j}\quad\mbox{if}\quad
1\le\t<\eta\le\infty.
$$

Stress again that in the definition of the Nikol'skii  space
$H_{p;j}^\a(\R^n)$ the order $r$ of the modulus of continuity is
{\it strictly greater} than the smoothness exponent $\a$. If
$\a\in \N,$ it is also natural to admit the value $r=\a.$ However,
it leads to
 completely different spaces -- Lipschitz type spaces.

 Assume that
$\a>0$ and denote by $\overline{\a}$ the least integer $s\ge \a.$
Let $1\le p<\infty$ and $1\le j\le n.$ Denote by
$\Lambda_{p;j}^\a(\R^n)$ the class of all functions $f\in
L^p(\R^n)$ such that
\begin{equation*}
||f||_{\l_{p;j}^\a}\equiv
\sup_{\d>0}\frac{\o_j^{\overline{\a}}(f;\d)_p}{\d^{\a}}<\infty.
\end{equation*}
Set also $||f||_{\Lambda_{p;j}^\a}=||f||_p+||f||_{\l_{p;j}^\a}.$

Clearly, $||f||_{\Lambda_{p;j}^\a}=||f||_{H_{p;j}^\a}$ if
$\a\not\in \N.$ If $\a\in \N,$ then we have the strict embedding
$\Lambda_{p;j}^\a\subset H_{p;j}^\a.$ Moreover, by the
Hardy--Littlewood theorem, if $\a\in \N,$ then
\begin{equation*}
\Lambda_{p;j}^\a(\R^n)=W_{p;j}^\a(\R^n)\quad\text{for $1<p\le
\infty$}.
\end{equation*}

\vskip 5pt

The strict embedding holds for $1\le p\le \infty$
$$
L_{p;j}^\a(\R^n)\subset \Lambda_{p;j}^\a(\R^n)=H_{p;j}^\a(\R^n),
\quad \a\not\in \N
$$

If $\a\in\N,$ then
$$
L_{1;j}^\a(\R^n)\equiv W_{1;j}^\a(\R^n)\subset
\Lambda_{1;j}^\a(\R^n)
$$
and
$$
L_{p;j}^\a(\R^n)\equiv W_{p;j}^\a(\R^n)=\Lambda_{p;j}^\a(\R^n)
\quad (1<p<\infty).
$$

\vskip 5pt

If $\a_j>0 ~~(j=1,...,n)$ and $1\le p <\infty,$ we set
$$
\Lambda_p^{\a_1,...,\a_n}(\R^n)=\bigcap_{j=1}^n
\Lambda_{p;j}^\a(\R^n).
$$
We shall call $\Lambda_p^{\a_1,...,\a_n}(\R^n)$ a Lipschitz space.

If all $\a_j$ are non-integer, then
$\Lambda_p^{\a_1,...,\a_n}(\R^n)=H_p^{\a_1,...,\a_n}(\R^n).$ If
all $\a_j$ are integer and $1<p<\infty,$ then
$$
\Lambda_p^{\a_1,...,\a_n}(\R^n)=L_p^{\a_1,...,\a_n}(\R^n)=W_p^{\a_1,...,\a_n}(\R^n).
$$
The most interesting  (and the most difficult) case is when among
the numbers $\a_j$ there are integers, but not all of them are
integers. In this case $\Lambda_p^{\a_1,...,\a_n}$ inherits partly
properties of the Sobolev spaces, and partly - properties of the
Nikolskii spaces.

\vskip 5pt

We will discuss  the problem of {\it embedding with limiting
exponent} for Lipschitz classes.

For any $\a_j>0$ and $1\le p\le \infty$ we have the following
embeddings
\begin{equation}\label{relat}
L_p^{\a_1,...,\a_n}(\R^n)\subset
\Lambda_p^{\a_1,...,\a_n}(\R^n)\subset H_p^{\a_1,...,\a_n}(\R^n).
\end{equation}
For $1\le p\le \infty,$ the right embedding in (\ref{relat})
becomes equality if and only if $\a_j\not\in \N$ for all
$j=1,...,n.$ In the left embedding equality takes place if and
only if $1<p\le \infty$ and $\a_j\in N, ~~ j=1,...,n.$

Let $n\ge 2$. Set
$$
\a\equiv n\left(\sum_{j=1}^n\frac1{\a_j}\right)^{-1}.
$$
Assume that $1\le p<\infty$ and $\a<n/p.$ Let $p^*=np/(n-\a p).$
Then
$$
L_p^{\a_1,...,\a_n}(\R^n)\subset L^{q}(\R^n) \quad\text{for all
$p<q\le p^*$}
$$
and
$$
H_p^{\a_1,...,\a_n}(\R^n)\subset L^q(\R^n) \quad\text{for all
$p<q< p^*$},
$$
but for $q=p^*$ the latter embedding does not hold. The problem
arises: what can be said about the embedding
$$
\Lambda_p^{\a_1,...,\a_n}(\R^n)\subset L^{p^*}(\R^n) ?
$$

The first result in this problem was obtained in our work
\cite{K1985}  for $0<\a_j\le 1.$
\begin{teo}\label{predq} Let $1\le p <\infty, ~~0<\a_j\le 1,$ and
$$
\a\equiv n\left(\sum_{j=1}^n\frac1{\a_j}\right)^{-1}<\frac np.
$$
Let $p^*=np/(n-\a p).$ Let $\nu$ be the number of $\a_j$ that are
equal to 1. The embedding
\begin{equation}\label{limq}
\Lambda_p^{\a_1,...,\a_n}(\R^n)\subset L^{p^*}(\R^n)
\end{equation}
holds if and only if
$$
\nu\ge \frac{n}{\a} - p.
$$
\end{teo}

\begin{rem} It follows that, in contrast to the Sobolev-Liouville
and Nikol'skii spaces, the embedding
$\Lambda_p^{\a_1,...,\a_n}\subset L^q$ is not uniquely determined
by the value of the harmonic mean $\a.$ Roughly speaking, this
means that for the spaces $\Lambda_p^{\a_1,...,\a_n}$ the
contribution of the variable $x_k$ is not proportional to
$1/\a_k.$
\end{rem}

Theorem \ref{predq} was extended by Netrusov \cite{Net1}
to arbitrary values of $\a_k>0.$ Moreover, Netrusov proved a
theorem on embedding of $\Lambda_p^{\a_1,...,\a_n}$ into Lorentz
spaces. He proposed another approach based on a modification of
the method of integral representations. However, his proof was
also long and complicated, and it did not work in the case $p=1.$
Applying rearrangements, we obtained \cite{K1998}  a new
proof of these results, including the case $p=1.$
\begin{teo}\label{netrusov} Let $n\ge 2$ and  $\a_j>0
~~(j=1,...,n).$ Let
$$
\a= n\left(\sum_{j=1}^n\frac1{\a_j}\right)^{-1}, ~~1\le
p<\frac{n}{\a},\quad \text{and} \quad p^*=\frac{np}{n-\a p}.
$$
Assume that there is an integer among the numbers $\a_j.$ Let
$$
\a'=\left(\sum_{j:\a_j\in \N}
\frac1{\a_j}\right)^{-1}\quad\text{and}\quad s=\frac{n\a'p}{\a}.
$$
Then for every function $f\in \Lambda_p^{\a_1,...,\a_n}(\R^n)$ we
have
\begin{equation}\label{netrus}
||f||_{p^*,s}\le c \sum_{j=1}^n||f||_{\l_{p;j}^{\a_j}}.
\end{equation}
\end{teo}

It was also proved by Netrusov that the index $s$ in this theorem
can not be replaced by a smaller one. Note that for a given value
of the mean index $\a,$ the bigger is the number of the integers
among $\a_j$ the smaller is the index $s.$ If there are no
integers $\a_j$ at all, then $s=\infty.$ In the other extreme
case, if all $\a_j$ are integers, we have $s=p$ and Theorem
\ref{netrusov} coincides with embedding theorem with limiting
exponent for anisotropic Sobolev spaces $W_p^{\a_1,...,\a_n}$
\cite{K1993}.

If $0<\a_j\le 1$ for all $j=1,...,n,$ then we have $s=np/(\nu\a),$
where $\nu$ is the number of $\a_j$ that are equal to 1. We have
$s\le p^*$ if and only if $\nu\ge n/\a-p.$ This is exactly the
necessary and sufficient condition for the embedding (\ref{limq})
(see Theorem \ref{predq}).

We observe that  {\it direct} estimates of rearrangements in terms of partial moduli of continuity are unknown;
it would be very interesting to find such estimates (see \cite{K1998}, \cite{K2007}).

\vskip 4pt

\section{Fourier transforms}

\vskip 4pt

\centerline{\bf Definition of the Fourier transform}

\vskip 4pt

Let $f$ be a function in $L^1(\R^n).$ The Fourier transform of $f$
is defined by
$$
\widehat{f}(\xi)=\int_{\R^n}f(x)e^{-i2\pi x\cdot \xi}\,dx, \quad
\xi\in \R^n.
$$

The Convolution theorem states that if $f,g\in L^1(\R^n),$ then
$$\widehat{f\ast g}(\xi)=\widehat{f}(\xi)\widehat{g}(\xi).$$

\vskip 4pt We have also the classical Plansherel's theorem.

\begin{teo}\label{Plansherel} If $f\in L^1(\R^n)\cap L^2(\R^n),$
then $\widehat{f}\in L^2(\R^n),$ and $||\widehat{f}||_2=||f||_2.$
\end{teo}
\vskip 4pt

Thus, the map $f\mapsto \widehat{f}$ is a bounded linear operator
defined on a dense subset $L^1\cap L^2$ of the space $L^2(\R^n);$
moreover, it is an isometry. If $f\in L^2(\R^n)$ (but $f\not\in
L^1(\R^n),$ there exists a sequence $\{f_k\}$ of functions in $
L^1(\R^n)\cap L^2(\R^n)$ such that $||f-f_k||_2 \to 0.$ We have
$$
||\widehat{f}_j-\widehat{f}_k||_2=||f_j-f_k||_2
$$
and hence $\{\widehat{f}_k\}$ is a Cauchy sequence in $L^2(\R^n)$.
Thus, $\{\widehat{f}_k\}$ converges to some function in
$L^2(\R^n)$, which we call $\widehat{f}.$ Obviously, $\widehat{f}$
does not depend on the choice of a sequence $\{f_k\}$. Moreover,
$$
||\widehat{f}||_2=\lim_{k\to \infty}||\widehat{f}_k||_2=\lim_{k\to
\infty}||f||_2=||f||_2.
$$

We observe also that $f\mapsto \widehat{f}$ is not only an
isometry but it is a {\it unitary transformation} of $L^2(\R^n)$
onto itself. It follows from the Fourier inversion theorem.
\begin{teo}\label{Inversion}
For any $f\in L^2(\R^n)$
$$
f(x)=\lim_{k\to\infty}\int_{|\xi|\le k}\widehat{f}(\xi)e^{i2\pi
x\cdot \xi}\,d\xi \quad\mbox({convergence~\, in}\quad L^2(\R^n)).
$$

\end{teo}

Thus, $f$ is the Fourier transform of $g(\xi)= \widehat{f}(-\xi).$

\vskip 4pt

We have defined the Fourier transform for functions in $L^1(\R^n)$
and functions in $L^2(\R^n)$. Now, let $f\in L^1(\R^n)+
L^2(\R^n),$ that is,
\begin{equation}\label{decomp}
f=f_1+f_2,\quad f_1\in L^1(\R^n), \,\,\, f_2\in L^2(\R^n).
\end{equation}
Set $\widehat{f}=\widehat{f_1}+\widehat{f_2}.$ In the case $f\in
L^1\cap L^2$ this definition coincides with the previous one. It
is easy   to show that definition does not depend on the choice of
decomposition (\ref{decomp}). Since $L^p(\R^n)\subset L^1(\R^n)+
L^2(\R^n)$ for $1\le p\le 2,$ the Fourier transform is defined for
all $f\in L^p(\R^n),\,\,\, 1\le p\le 2.$

\vskip 8pt

\centerline{\bf The Hausdorff -- Young inequality}

\vskip 8pt
\begin{teo}\label{Haus-Young}
Let $f\in L^p(\R^n),\,\,\, 1\le p\le 2.$ Then
\begin{equation}\label{haus-young}
||\widehat{f}||_{p'}\le ||f||_p.
\end{equation}
\end{teo}

This theorem was proved by W. Young in 1913 in the case when $p'$
is an even integer. Namely, Young observed that in this case the
Fourier transform inequality can be obtained from the convolution
inequality. For example, for $p=4/3,\,\, p'=4$
$$
\int_{\R^n}|\widehat{f}(\xi)|^4
d\xi=\int_{\R^n}|\widehat{f}(\xi)^2|^2 d\xi
$$
$$
=\int_{\R^n}|\widehat{f\ast f}(\xi)|^2 d\xi=||f\ast f||_2^2\le
||f||_{4/3}^4.
$$

In 1923 Hausdorff proved inequality (\ref{haus-young}) for all
$p\in [1,2].$

In the periodic case the Hausdorff -- Young theorem states the
following.

\begin{teo}\label{Periodic} Let $f\in L^p[0,2\pi],$  $\,1\le p\le
2,$ and
$$
c_n=\frac{1}{2\pi}\int_0^{2\pi} f(x)e^{-inx}dx\quad (n\in \Z).
$$
Then
\begin{equation}\label{periodic}
\left(\sum_{n\in \Z} |c_n|^{p'}\right)^{1/p'}\le
\left(\frac{1}{2\pi}\int_0^{2\pi} |f(x)|^p\,dx\right)^{1/p}.
\end{equation}

\end{teo}

In 1923 F. Riesz extended this theorem to an arbitrary uniformly
bounded orthonormal system. In 1926 M. Riesz gave an alternative
proof of this result as one of applications of his convexity
theorem.

For $p>2$ Theorems \ref{Haus-Young} and \ref{Periodic} fail. For
example, let $p>2$ and let a sequence $\{c_k\}$ of positive
numbers be such that
$$
\sum_{k=1}^\infty c_k^2<\infty, \quad\mbox{but} \quad
\sum_{k=1}^\infty c_k^{p'}=\infty.
$$
Then the function
$$
f(x)=\sum_{k=1}^\infty c_k \cos 2^k x
$$
belongs to $L^p[0,2\pi]$, but inequality (\ref{periodic}) doesn't
hold.

\vskip 4pt

Inequality (\ref{periodic}) is sharp for all $1\le p\le 2.$
Indeed, it becomes equality  for functions $f(x)=Ae^{i2\pi mx}.$
Hardy and Littlewood proved that equality in (\ref{periodic}) is
attained only for such exponential functions.

However, inequality (\ref{haus-young}) for $1<p<2$  can be
improved. First it was proved by K. Babenko \cite{Bab} in 1961 for
$p'=2,4,6,....$ For all $p\in (1,2)$ Beckner \cite{Beckner} in
1975 proved the inequality
\begin{equation}\label{bruc}
||\widehat{f}||_{p'}\le A_p^n||f||_p,
\end{equation}
where
$$
A_p=\left(\frac{p^{1/p}}{(p')^{1/p'}}\right)^{1/2}
$$

For the gaussian $\operatorname{exp(-\pi |x|^2})$ (\ref{bruc})
becomes equality.

\vskip 8pt

\centerline{\bf Hardy -- Littlewood -- Paley inequality}

\vskip 8pt

First, the following Hardy -- Littlewood type theorem holds.

\begin{teo}\label{Trans_1} If $f\in L^p(\R^n),\,\,\, 1< p\le 2,$
then
\begin{equation}\label{trans_1}
\left(\int_{\R^n} |\xi|^{n(p-2)}
|\widehat{f}(\xi)|^pd\xi\right)^{1/p}\le c||f||_p.
\end{equation}
\end{teo}

A stronger inequality is given by the Hardy -- Littlewood -- Paley
theorem.

\begin{teo}\label{Trans_2}  If $f\in L^p(\R^n),\,\,\, 1< p\le 2,$
then
\begin{equation}\label{trans_2}
\left(\int_0^\infty t^{p-2} \widehat{f}^*(t)^p dt\right)^{1/p}\le
c||f||_p.
\end{equation}
\end{teo}

This theorem gives a refinement of Theorem \ref{Trans_1}. Indeed,
as we have already observed, for the function $\f(x)=1/|x|$ $(x\in
\R^n))$ we have $\f^*(t)=(v_n/t)^{1/n}$. Thus, by the Hardy --
Littlewood inequality,
$$
\int_{\R^n} |\xi|^{n(p-2)} |\widehat{f}(\xi)|^pd\xi\le
\int_0^\infty t^{p-2} \widehat{f}^*(t)^p dt.
$$

Further, the left-hand side of (\ref{trans_2}) is exactly the
Lorentz norm $||\widehat{f}||_{p',p}.$ Recall that we have a
strict embedding $L^{p',p}\subset L^{p'}, \,\, \, 1<p\le 2.$ Thus,
inequality (\ref{trans_2}) implies the Hausdorff -- Young
inequality, but with additional constant on the right-hand side
(which blows up as $p\to 1$).

\vskip 6pt

We observe that initially Theorems \ref{Trans_1} and \ref{trans_2}
were proved by Hardy and Littlewood \cite{HL2}, \cite{HL3} (1927, 1931) for the trigonometric Fourier series. In particular, in the
periodic case Theorem \ref{Trans_1} states the following.
\begin{teo}\label{Periodic_2} Let $f\in L^p[0,2\pi],$  $1< p\le
2,$ and
$$
c_n=\frac{1}{2\pi}\int_0^{2\pi} f(x)e^{-inx}dx\quad (n\in \Z).
$$
Then
\begin{equation}\label{periodic_2}
\left(\sum_{n\in \Z} |c_n|^{p}(|n|+1)^{p-2}\right)^{1/p}\le
A||f||_p.
\end{equation}

\end{teo}

In 1931 Paley \cite{Paley} extended  this theorem to the Fourier
series with respect to  arbitrary uniformly bounded orthonormal
system $\{\f_n\}$ on $[0,1]$. From this result he derived a
rearrangement inequality for the  Fourier coefficients $a_n$
$$
\left(\sum_{n=1}^\infty (a_n^*)^{p}n^{p-2}\right)^{1/p}\le
A||f||_p.
$$

\vskip 6pt Of course, these theorems fail for $p=1.$ For example,
the series
$$
\sum_{n=1}^\infty \frac{\cos nx}{\log (n+1)}
$$
is the Fourier series of some integrable function $f$, but
$$
\sum_{n=1}^\infty \frac{a_n}{n}=\sum_{n=1}^\infty \frac{1}{n\log
(n+1)}=\infty.
$$
\vskip 6pt

In 1937, Pitt \cite{Pitt} proved the following theorem.
\begin{teo}\label{Periodic_3} Let  $1< p\le q<\infty,$ $0\le \a<1/p',$ and
$\l=1/q+1/p-1+\a\ge 0.$ Let  $f\in L^1[0,2\pi]$ and
$$
c_n=\frac{1}{2\pi}\int_0^{2\pi} f(x)e^{-inx}dx\quad (n\in \Z).
$$
Then
\begin{equation}\label{periodic_3}
\left(\sum_{n\in \Z} |c_n|^{q}(|n|+1)^{-\l q}\right)^{1/q}\le
A\left(\int_{-\pi}^{\pi}|f(x)|^p |x|^{\a p}dx\right)^{1/p}.
\end{equation}

\end{teo}

In 1956, Stein \cite{St1} extended this result to arbitrary
uniformly bounded systems on $[0,1]$. Under the same conditions on
$p,q,$ and $\a$, he obtained the following rearrangement
inequality
$$
\left(\sum_{n=1}^\infty (a_n^*)^{q}n^{-\l q}\right)^{1/q}\le
\left(\int_{0}^1 f^*(x)^p x^{\a p}dx\right)^{1/p}.
$$

\vskip 6pt

Suppose now that $f\in L^{p,r}(\R^n)\,\, (1<p<2, \,\, 1\le r\le
\infty).$ Then $f\in L^1(\R^n)+ L^2(\R^n)$ and hence, $\widehat{f}$
is defined. Moreover, the following inequality holds.

\begin{teo}\label{p-r} If $f\in L^{p,r}(\R^n)\,\, (1<p<2, \,\, 1\le r\le
\infty),$ then $\widehat{f}\in L^{p',r}(\R^n),$ and
$$
||\widehat{f}||_{p',r}\le c||f||_{p,r}.
$$
\end{teo}

This theorem was obtained by Herz \cite{Herz} in 1968 with the use
of the Marcinkiewicz interpolation theorem.

Let us consider the case  $p=2.$ In this case,  Herz \cite{Herz}
showed that the Fourier transform maps $L^{2,r}$ continuously into
$L^{2,q}$ if and only if $r\le 2 \le q.$ The positive part of this
statement is obvious. By duality, to obtain the negative part, it
is sufficient to show that, whatever be $1\le q<2$,  the Fourier
transform doesn't map $L^{2,1}$ continuously into $L^{2,q}$. To
prove it, Herz used the so called {\it lacunary functions}. For
the simplicity, we consider the periodic case. We show that for
any $1\le q<2$ there exists a function $f\in L^{2,1}$ such that
the sequence of its Fourier coefficients doesn't belong to
$l^{2,q}.$ Let $1\le q<2$. Set
$$
f(x)=\sum_{k=2}^\infty\frac{\cos 2^k x}{\sqrt{k}(\ln k)^{1/q}}.
$$
Denote for $n\in\N$
$$
a_n=\begin{cases} (\sqrt{k}(\ln k)^{1/q})^{-1}&\text{if $n=2^k,
~~k\ge
2$},\\
0&\text{otherwise}.
\end{cases}
$$
Since the sequence $\{a_n\}$ belongs to $l^2$, the function $f$
belongs to $L^p$ for any $1<p<\infty.$ However, $\{a_n\}$  doesn't
belong to $l^{2,q}.$

\vskip 6pt

We mention also the paper by Bochkarev \cite{Boch}. In this paper,
the Fourier coefficients with respect to uniformly bounded
orthonormal systems were studied. First, it was proved the
following
\begin{teo}\label{Bochk} Let $\{\f_n\}$ be an orthonormal system of functions
on$[0,1]$ such that
$$
||\f_n||_\infty\le M, \quad n\in \N.
$$
Let $f\in L^1[0,1]$ and let $c_n$ be Fourier coefficients of $f$
with respect to $\{\f_n\}$. Then for any $2<q\le\infty$ and $n\ge
2,$
\begin{equation}\label{bochka}
\sum_{k=1}^n (c_k^*)^2\le AM^2(\log n)^{1-2/q}||f||_{2,q}^2.
\end{equation}
\end{teo}

\vskip 4pt

Bochkarev showed that  inequality (\ref{bochka}) cannot be
improved.

\vskip 4pt We observe  that John Benedetto and Hans Heinig
\cite{BeHe}  studied Fourier transform inequalities in
weighted Lorentz spaces. Their approach was based on estimates of
rearrangements. \vskip 8pt

\centerline{\bf Rearrangement inequalities for Fourier transforms}
\vskip 8pt The following theorem was proved by Jurkat  and Sampson
\cite{JuSa} (1984).

 \vskip 6pt
\begin{teo}\label{Jurkat} Let $f\in L^1(\R^n)+ L^2(\R^n)$. Then
\begin{equation}\label{Ju-Sam}
\left(\int_0^t \widehat{f}^*(u)^2 du\right)^{1/2}\le
t^{1/2}\int_0^\tau f^*(s)ds+\left(\int_\tau^\infty f^*(s)^2
ds\right)^{1/2}
\end{equation}
for any $t,\tau>0.$
\end{teo}
\begin{proof} Choose a set $E\subset \R^n$ of measure $\tau$ such
that
$$
\{x: |f(x)|>f^*(\tau)\}\subset E\subset \{x: |f(x)|\ge
f^*(\tau)\}.
$$
Let $g=f\chi_E, \,\, h=f-g.$ Then
\begin{equation}\label{Ju-Sam_2}
||g||_1=\int_0^\tau f^*(s)ds,\quad ||h||_2=\left(\int_\tau^\infty
f^*(s)^2 ds\right)^{1/2}.
\end{equation}
We have also
$$
||\widehat{g}||_\infty\le ||g||_1\quad\mbox{and}\quad
||\widehat{h}||_2=||h||_2.
$$
Observe that $\widehat{f}^*(u)\le ||\widehat{g}||_\infty +
\widehat{h}^*(u).$ Thus,
$$
\left(\int_0^t \widehat{f}^*(u)^2 du\right)^{1/2}\le
t^{1/2}||\widehat{g}||_\infty+||\widehat{h}||_2.
$$
Applying (\ref{Ju-Sam_2}), we obtain (\ref{Ju-Sam}).

\end{proof}
\vskip 6pt
\begin{cor}\label{Jurkat_2} Let $f\in L^1(\R^n)+ L^2(\R^n)$. Then
\begin{equation}\label{Ju-Sam_3}
\int_0^t \widehat{f}^{*}(u)^2 du\le c\int_0^t\left(\int_0^{1/u}
f^*(s)ds\right)^2\,du= c\int_{1/t}^\infty f^{**}(u)^2\, du
\end{equation}
for any $t>0.$
\end{cor}

\vskip 6pt

From these results Jurkat and Sampson derived some weighted upper
estimates for $\widehat{f}$ (in particular, inequalities in
Lorentz norms). For this, they  used the following {\it convexity
principle} (which represents a generalization of Hardy, Littlewood
and P\'olya theorem).
\begin{teo}\label{Convexity} Assume that $\Psi(s)\ge 0$ is increasing and
convex for $s\ge 0,$ and $\Phi(s)\ge 0$ is increasing and concave
for $s\ge 0.$ Let $U(t),\, G(t),$ $H(t)$ be nonnegative and
measurable for $t>0.$ If $G$ decreases, then
\begin{equation}\label{Ju-Sam_4}
\int_0^t G(s)U(s)ds\le \int_0^t H(s)U(s)ds\quad\mbox{for all}\quad
t>0
\end{equation}
implies
\begin{equation*}%\label{Ju-Sam_5}
\int_0^t \Psi(G(s))U(s)ds\le \int_0^t
\Psi(H(s))U(s)ds\quad\mbox{for all}\quad t>0.
\end{equation*}
If $H$ increases, then (\ref{Ju-Sam_4}) implies
\begin{equation*}%\label{Ju-Sam_5}
\int_0^t \Phi(G(s))U(s)ds\le \int_0^t
\Phi(H(s))U(s)ds\quad\mbox{for all}\quad t>0.
\end{equation*}
\end{teo}

\vskip 6pt Rearrangement inequality (\ref{Ju-Sam_3}) actually was
proved in 1971 by Jodeit and Torchinsky \cite{JT}. More exactly,
they proved the following theorem.
\begin{teo}\label{Jod_Tor} Let $T$ be a sublinear operator. Then
the operator $T$ is of type $(1,\infty)$ and $(2,2)$ if and only
if for some constant $K$ and each $f\in L^1(\R^n)+ L^2(\R^n)$
$$
\int_0^t (Tf)^*(u)^2\,du\le K\int_{1/t}^\infty f^{**}(u)^2\, du,
\quad t>0.
$$
\end{teo}
\vskip 8pt

\centerline{\bf Estimates of Fourier transforms in Sobolev spaces}

\vskip 8pt

Let $r, n\in\N$ and $1< p\le 2$. Assume that $f\in W_p^r(\R^n).$
We have
$$
|\widehat{D_k^r f}(\xi)|=(2\pi |\xi_k|)^r |\widehat{f}(\xi)|
$$
and thus
\begin{equation}\label{6}
|\widehat{f}(\xi)|\asymp |\xi|^{-r}\sum_{k=1}^n |\widehat{D_k^r
f}(\xi)|.
\end{equation}
By the Hardy-Littlewood inequality, we have
\begin{equation}\label{f-sobolev_1}
\left(\int_{\R^n} |\xi|^{pr+n(p-2)}|\widehat{f}(\xi)|^p\,d\xi)\right)^{1/p}\le c\sum_{k=1}^n||D_k^r f||_p.
\end{equation}

If $f, g\in S_0(\R^n),$ then $(fg)^*(t)\le f^*(t/2)g^*(t/2)$ (see,
e.g., \cite{KPS}).  Further, recall that the rearrangement of the
function $\f(\xi)=1/|\xi|$ is equal to $(v_n/t)^{1/n}$. Thus,
using (\ref{6}) and applying Hardy-Littlewood-Paley inequality to
the derivatives, we obtain
\begin{equation}\label{f-sobolev_2}
\left(\int_0^\infty t^{pr/n+p-2}\widehat{f}^*(t)^p\,dt\right)^{1/p}\le c\sum_{k=1}^n||D_k^r f||_p.
\end{equation}
We observe that the right-hand sides of (\ref{f-sobolev_1}) and
(\ref{f-sobolev_2}) contain only the norms of non-mixed
derivatives. But we know (Theorem \ref{Smith}) that for
$1<p<\infty$
$$
\sum_{|s|=r}||D^s f||_p\le c\sum_{k=1}^n||D_k^r f||_p.
$$
Stress that this estimate fails for $p=1.$

For $p=1$ the Hardy-Littlewood inequality does not hold and it is impossible to use the reasonings described above.
In  fact, for the Fourier transforms of functions in $W^1_1(\R)$, the only available estimates are those based upon the obvious inequality
$||\widehat{g}||_\infty\le ||g||_1.$

On the other hand, by Hardy's inequality, for any $f\in H^1(\R^n)$\newline$(n\in\N)$
$$
\int_{\R^n} \frac{|\widehat{f}(\xi)|}{|\xi|^n}\,d\xi\le c||f||_{H^1}.
$$

\vskip 4pt

 It was first discovered by Bourgain \cite{Bour1} in 1981 that for
$n\ge 2$ the Fourier transforms of the derivatives of the
functions in Sobolev space $W_1^1(\R^n)$ satisfy Hardy's
inequality. More exactly, Bourgain considered the periodic case.
His studies were continued by Pelczy\'nski and  Wojciechowski
\cite{PeWo}. First, we have the following theorem (Bourgain; Pelczy\'nski and  Wojciechowski).

\vskip 4pt
\begin{teo}\label{Pelcz} Let $f\in W_1^r(\R^n)$ $(n\ge 2, \,
r\in\N).$ Then
\begin{equation}\label{pelcz}
\int_{\R^n} |\widehat{f}(\xi)||\xi|^{r-n}\,d\xi\le c
\sum_{|s|=r}||D^s f||_1.
\end{equation}
\end{teo}
\vskip 4pt

Equivalently,
$$
\sum _{|s|=r}\int_{\mathbb{R}^n}\frac{|\widehat{D^s
f}(\xi)|}{|\xi|^n}d\xi \le c \sum _{|\alpha|=r}\|D^\alpha f\|_1.
$$

This is  Hardy type inequality. Bourgain \cite{Bour2} in 1985 and
Pelczy\'nski and  Wojciechowski \cite{PeWo} in 1993 proved also the following
theorem.

\begin{teo}\label{Bourg} If $f\in W_1^1(\R^n)$ $(n\ge 2)$, then
\begin{equation}\label{bourg}
\sum_{\nu\in
\Z}\left(\int_{D_\nu}|\widehat{f}(\xi)|^n\,d\xi\right)^{1/n}\le
c\sum_{k=1}^n||D_k^r f||_1,
\end{equation}
where
$$
D_\nu=\{\xi\in \R^n: 2^{\nu-1}< |\xi|\le 2^\nu\} \quad (\nu\in
\Z).
$$
\end{teo}
\vskip 4pt
We observe that  by Sobolev-Gagliardo-Nirenberg
theorem, $$W_1^1(\R^n)\subset L^{n'}(\R^n),\quad n\ge 2.$$
Applying Hausdorff-Young inequality, we obtain that for any  $f\in
W_1^1(\R^n)$ its Fourier transform belongs to $L^n(\R^n).$ Of
course, (\ref{bourg}) gives a stronger statement. \vskip 4pt

\vskip 4pt We return to Theorem \ref{Pelcz}. In comparison with
inequality (\ref{f-sobolev_1}), the right-hand side of
(\ref{pelcz}) contains all the derivatives of order $r.$ The
$L^1-$norms of mixed derivatives cannot be estimated by the
$L^1-$norms of pure derivatives. Nevertheless, we proved
\cite{K1997}  that the  right-hand side of inequality
(\ref{pelcz}) can be replaced by the sum of non-mixed derivatives
of the same order.

\begin{teo}\label{Kol97} Let $r_1,...,r_n\in \N$ $(n\ge 2)$ and
let
$$
r=n\left(\sum_{k=1}^n \frac{1}{r_k}\right)^{-1}.
$$
If
 $f\in W_1^{r_1,\dots,r_n} (\mathbb{R}^n)\,\,\, (n\ge 2)$,
 then
\begin{equation}\label{kol_1}
\int_{\mathbb{R}^n}|\widehat{f}(\xi)|\left(\sum_{k=1}^n|\xi_k|^{r_j}\right)^{1-n/r}d\xi
 \le
c \sum _{k=1}^n\|D_k^{r_k} f\|_1.
\end{equation}
\end{teo}
\vskip 4pt

In the isotropic case $r_1=\cdots =r_n=r$ inequality (\ref{kol_1}) has the form
\begin{equation}\label{kol_2}
\int_{\mathbb{R}^n}|\widehat{f}(\xi)|\xi|^{r-n}d\xi
 \le
c \sum _{k=1}^n\|D_k^{r} f\|_1.
\end{equation}
In contrast to (\ref{pelcz}), the right-hand side of (\ref{kol_2}) contains only the norms of non-mixed derivatives.

For any $s=(s_1,...,s_n)$ with $|s|=r$
$$
|\widehat{D^s f}(\xi)|=(2\pi)^r|\widehat{f}(\xi)|\prod_{j=1}^n
|\xi_j|^{s_j}\le (2\pi |\xi|)^r|\widehat{f}(\xi)|.
$$
Thus, we obtain
\begin{teo}
 Let $f\in C_0^\infty(\mathbb{R}^n)~(n\ge 2)$. Then for any  $r\in \mathbb{N}$
$$
\sum _{|s|=r}\int_{\mathbb{R}^n}\frac{|\widehat{D^s f}(\xi)|}{|\xi|^n}d\xi
\le c \sum _{k=1}^n\|D_k^r f\|_1.
$$
\end{teo}

\vskip 4pt In \cite{K1997}, we proved also the following rearrangement inequality.
\begin{teo}\label{Kol97_2} Let $r_1,...,r_n\in \N$ $(n\ge 2)$ and
let
$$
r=n\left(\sum_{k=1}^n \frac{1}{r_k}\right)^{-1}<n.
$$
If
 $f\in W_1^{r_1,\dots,r_n} (\mathbb{R}^n)\,\,\, (n\ge 2)$, then
\begin{equation}\label{kol_3}
||\widehat{f}||_{n/r,1}
 \le
c \sum _{k=1}^n\|D_k^{r_k} f\|_1.
\end{equation}
\end{teo}
\vskip 4pt Observe that in conditions of this theorem we have
\begin{equation}\label{kol_4}
||f||_{n/(n-r),1}
 \le
c \sum _{k=1}^n\|D_k^{r_k} f\|_1
\end{equation}
(it was proved in our work \cite{K1993}). If $n/(n-r)<2$
(for example, if $n\ge 3$ and $r=1$), then (\ref{kol_3}) follows
immediately from (\ref{kol_4}), if we apply the inequality
$$
||\widehat{f}||_{p',r}\le ||f||_{p,r} \quad (1<p<2, \,\, 1\le r\le
\infty).
$$

These arguments cannot be applied, for example, to the space
$W_1^1(\R^2).$ We have $W_1^1(\R^2)\subset L^{2,1}(\R^2)$, and
this embedding is sharp. As we know, the condition $f\in
L^{2,1}(\R^2)$ does not imply that  $\widehat f\in L^{2,1}(\R^2)$.
However, by Theorem \ref{Kol97_2}, for functions in $W_1^1(\R^2)$
their Fourier transforms belong to $ L^{2,1}(\R^2)$.

\vskip 8pt

In our work \cite{K2001} we extended Theorems \ref{Kol97}
and \ref{Kol97_2} to the fractional Sobolev spaces
$L_p^{\a_1,...,\a_n}(\R^n).$
In this work we proved also the following theorem.
\begin{teo}\label{Product} Assume that $r_j\in\N ~~(j=1,...,n;
ñ\ge 2)$. Then
$$
\left(\int_{R^n}|\widehat{f}(\xi)|^n\prod_{j=1}^n|\xi_j|^{r_j-1}\,d\xi\right)^{1/n}
\le c \prod_{j=1}^n||D_j^{r_j}f||_1
$$
for each $f\in W_1^{r_1,...,r_n}(\R^n).$
\end{teo}
\vskip 4pt For $n=2$ this result was obtained by other methods by
Pelczy\'nski and Senator \cite{PeSe}.

\vskip 6pt

 The proofs of these results can be subdivided in two
parts. The main (and much more involved part) is the following
theorem which we have already formulated above.

For $r>0,$ let $\bar r$ be the least integer such that $r\le
\bar{r}.$

\begin{teo} Let $r_1,...,r_n$ be positive numbers and let
$$
r=n\left(\sum_{j=1}^n\frac{1}{r_j}\right)^{-1}, \,\, 1\le
p<q<\infty, \,\, \varkappa=1-\frac nr\left(\frac 1p-\frac
1q\right)>0.
$$
Set $\a_j=\varkappa r_j \,\,(j=1,...,n).$ If $1<p<\infty$ and
$n\ge 1$, or $p=1$ and $n\ge 2,$ then for  every function $f\in
L^{r_1,...,r_n}_p(\R^n)$
$$
\sum_{j=1}^n\left(\int_0^\infty
[h^{-\a_j}\o_j^{\bar{r}_j}(f;h)_q]^p\frac{dh}{h}\right)^{1/p}\le c
\sum_{j=1}^n||D_j^{r_j}f||_p.
$$
\end{teo}

\vskip 4pt

For $r_1=...=r_n=1$ this theorem was proved in our work
\cite{K1988} in 1988. Anisotropic case was studied in our works
\cite{K1993}  and \cite{K2001}.

\vskip 4pt

The other part of the proofs was based on the estimates via moduli
of continuity.

\vskip 8pt

\centerline{\bf Estimates of Fourier transforms}

\centerline {\bf in terms of moduli of continuity}

\vskip 8pt

Apparently the first estimates of this type were obtained by S.
Bernstein in 1914. Bernstein proved that for each periodic
function $f\in \operatorname{Lip} \a$ ($\a>1/2$) the sequence of
its Fourier coefficients belongs to $l^1,$ and this property fails
for $\a=1/2.$

\vskip 8pt

The following lemma was proved in our paper \cite{K1997}.

\begin{lem}\label{mod_estim} Let $r_1,...,r_n\in \N$ $(n\in\N)$,
$$
r=n\left(\sum_{j=1}^n \frac{1}{r_j}\right)^{-1},
\quad\mbox{and}\quad s_j=\frac{r}{n r_j}.
$$
Let $f\in L^p(\R^n),$ $1\le p\le 2.$ Then
\begin{equation}\label{kol_10}
\widehat{f}^*(t)\le c
t^{1/p-1}\sum_{j=1}^n\o_j^{r_j}(f;t^{-s_j})_p \quad(t<0).
\end{equation}
\end{lem}
\begin{proof} Set $\f_{h,j}(x)=\Delta_j^{r_j}(h)f(x)\,\, (h>0).$
Let $E\subset \R^n$ be a set of measure $t$ such that
$$
|\widehat{f}(\xi)|\ge \widehat{f}^*(t)\quad\mbox{for any} \quad
\xi\in E.
$$
Further, let
$$
A_j=\{\xi: |\xi_j|\ge \frac{t^{s_j}}{2}\}, \quad j=1,...,n.
$$
We have $\sum_{j=1}^n s_j=1.$ Thus,
$$
\left|\left(\bigcup_{j=1}^n
A_j\right)^c\right|=\left|\bigcap_{j=1}^n
A_j^c\right|=\frac{t}{2^n}.
$$
This implies that $|E\cap(\cup_{j=1}^n A_j)|\ge t/2.$ Hence, there
exists $j=j(t)$ such that $|E\cap A_j|\ge t/(2n).$ Set $Q=E\cap
A_j.$ We have
$$
\widehat{\f_{h,j}}(\xi)= \widehat{f}(\xi)\sigma(h\xi_j),
\quad\mbox{where}\quad \sigma(u)=(e^{2\pi ui}-1)^{r_j}.
$$
Set $\d=r_j t^{-s_j}.$ We show that
\begin{equation*}
\frac{1}{\d}\int_0^\d|\sigma(h\xi_j)|\,dh>\frac12\quad\mbox{for
any}\quad \xi\in Q.
\end{equation*}
Indeed,
$$
|\s(u)\ge (1-\cos 2\pi u)^{r_j}\ge 1-r_j\cos 2\pi u.
$$
Thus, if $|\l\ge t^{s_j}/2,$ then
$$
\frac{1}{\d}\int_0^\d|\sigma(\l h)|\,dh\ge 1-\frac{r_j\sin 2\pi
\l\d}{2\pi \l\d}\ge 1-\frac1\pi.
$$

 Now we have
 $$
 \begin{aligned}
\frac{1}{\d}\int_0^\d dh\int_Q|\widehat{\f_{h,j}}(\xi)|d\xi
&=\frac{1}{\d}
dh\int_0^\d\int_Q|\widehat{f}(\xi)\sigma(h\xi_j)|d\xi\\
&\ge \frac12|Q|\widehat{f}^*(t)\ge \frac{t}{4n}widehat{f}^{*}(t).
\end{aligned}
$$
On the other hand, by the Hausdorff-Young inequality,
$$
\begin{aligned}
\int_Q|\widehat{\f_{h,j}}(\xi)|d\xi &\le
|Q|^{1/p}\left(\int_Q|\widehat{\f_{h,j}}(\xi)|^{p'}d\xi\right)^{1/p'}\\
& \le |Q|^{1/p}||\f_{h,j}||_p\le t^{1/p}||\f_{h,j}||_p.
\end{aligned}
$$
Therefore,
$$
\widehat{f}^*(t)\le c t^{1/p-1}\frac{1}{\d}\int_0^\d
||\f_{h,j}||_p dh\le c' t^{1/p-1}\o_j^{r_j}(f;t^{-s_j})_p.
$$

\end{proof}

\vskip 4pt

Assume that $\a_j>0$ and $1\le p,\t<\infty.$ Let $r_j$ be the
least integer such that $r_j>\a_j.$ Let
$$
||f||_{b_{p,\t}^{\a_1,...,\a_n}}=\sum_{j=1}^n\left(\int_0^{\infty}\left(t^{-\a_j}\o_k^{r_j}(f;t)_p\right)^\t
\frac{dt}{t}\right)^{1/\t}
$$
denote the homogenous Besov norm.

\vskip 4pt

 Applying Lemma \ref{mod_estim}, we immediately obtain

\begin{teo}\label{Four_1} Let $\a_j>0$ $(j=1,...,n)$, $1\le p\le
2,$ $1\le\t\le\infty,$
$$
\a=n\left(\sum_{j=1}^n\frac1{\a_j}\right)^{-1},
\quad\mbox{and}\quad q=\left(\frac{\a}{n}+\frac1{p'}\right)^{-1}.
$$
If $f\in B_{p,\t}^{\a_1,...,\a_n}(\R^n),$ then $\widehat{f}\in
L^{q,\t}(\R^n),$ and
$$
||\widehat{f}||_{q,\t}\le c ||f||_{b_{p,\t}^{\a_1,...,\a_n}}.
$$
\end{teo}
\vskip 6pt

In the case $\a_1=...=\a_n=\a$ this result was obtained by other
methods by A. Pietsch \cite[p. 270]{Pi} (1980). Taking $\t=1$ and
$\a=n/p,$ we obtain the Bernstein-Szasz theorem:  if $f\in
B_{p,1}^{n/p}(\R^n)\,\,\, (1\le p\le 2),$ then $\widehat{f}\in
L^1(\R^n).$

\vskip 6pt

Return to the inequality
 \begin{equation}\label{Question}
||\widehat{f}||_{n/r,1}
 \le
c \sum _{k=1}^n\|D_k^{r_k} f\|_1,\quad r=n\left(\sum_{k=1}^n
\frac{1}{r_k}\right)^{-1}<n.
\end{equation}
As we have already mentioned, the proof of this inequality was
obtained by the combination of estimates of the rearrangement
$\widehat{f}^*(t)$ via moduli of continuity in some $L^p, \,\,
1<p<2,$ and estimates of these moduli of continuity via norms of
partial derivatives in $L^1.$ It would be interesting to find {\it
direct estimates} of $\widehat{f}^*(t)$ involving rearrangements
of the derivatives (and, maybe, rearrangement of $f$) and yielding
(\ref{Question}). We observe that Jurkat-Sampson inequality
$$
\int_0^t \widehat{\f}^{*}(u)^2 du\le c\int_0^t\left(\int_0^{1/u}
\f^*(s)ds\right)^2\,du= c\int_{1/t}^\infty \f^{**}(u)^2\, du
$$
applied to $\f=D_k^{r_k} f$ cannot work since at the right-hand
side we have the operator $\f^{**}$ which is unbounded in $L^1.$

In our joint work with J. Garc\'{i}a-Cuerva \cite{GCK}  we
obtained sharp weighted rearrangement estimates of Fourier
transforms in terms of moduli of continuity (we studied isotropic
case).

\vskip 8pt

Finally, we observe that  the relations between moduli of
smoothness of functions and  growth properties of their Fourier
transforms were studied in many works, especially in the last 20
years. We refer to the works of Benedetto and Heinig \cite{BeHe}
and Gorbachev and Tikhonov \cite{GoTi} and references therein.

\end{document}